\documentclass{article}
\usepackage{amssymb, amsthm, fancyhdr,amsmath,enumerate,graphicx,multicol,psfrag}
\usepackage{subcaption}
\usepackage[inline]{enumitem}
\usepackage{color}
\usepackage{mathrsfs}
\usepackage{bbm}
\usepackage{appendix}
\usepackage[colorlinks=true,linkcolor=blue]{hyperref}
\usepackage{tcolorbox}
\usepackage{mathtools} 
\mathtoolsset{showonlyrefs=true}  
\usepackage[left=1.2in,top=1in,right=1.2in,bottom=1in]{geometry}
\usepackage{tcolorbox}

\usepackage{tikz-cd}

\DeclareMathOperator{\supp}{supp}

\DeclareMathOperator{\diam}{diam}

\DeclareMathOperator{\av}{AV}
\DeclareMathOperator{\tv}{TV}
\DeclareMathOperator{\cpl}{Cpl}
\DeclareMathOperator{\cplbc}{Cpl_{bc}}
\DeclareMathOperator{\id}{id}
\DeclareMathOperator{\Id}{I}
\DeclareMathOperator{\law}{Law}
\DeclareMathOperator{\per}{Per}

\def\a{\mathcal{A}}

\def\n{\mathcal{N}}
\def\sp{\mathcal{P}}

\def\x{\mathcal{X}}
\def\y{\mathcal{Y}}

\def\w{\mathcal{W}}

\def\C{\mathbb{C}}

\def\E{\mathbb{E}}

\def\L{\mathcal{L}}

\def\N{\mathbb{N}}

\def\P{\mathbb{P}}

\def\R{\mathbb{R}}

\def\nab{\nabla}
\def\ep{\varepsilon}

\def\ep{\varepsilon}

\def\st{\;\vert\;}

\newlist{anumerate}{enumerate}{1}
\setlist[anumerate,1]{label=(\alph*)}

\newcommand{\abs}[1]{\left\vert#1\right\vert}
\newcommand{\norm}[1]{\left\Vert#1\right\Vert}

\newcommand{\ind}[1]{\mathbbm{1}_{#1}}

\newtheorem{thm}{Theorem}
\newtheorem{cor}[thm]{Corollary}
\newtheorem{df}[thm]{Definition}
\newtheorem{assume}[thm]{Assumption}
\newtheorem{prop}[thm]{Proposition}
\newtheorem{rmk}[thm]{Remark}
\newtheorem{lem}[thm]{Lemma}
\newtheorem{ex}[thm]{Example}

\begin{document}

\title{Bounding adapted Wasserstein metrics}
\author{Jose Blanchet\thanks{Stanford Univeristy, Management Science and Engineering, \href{mailto:jose.blanchet@stanford.edu}{jose.blanchet@stanford.edu}} \and Martin Larsson\thanks{Carnegie Mellon University, Department of Mathematical Sciences, \href{mailto:larsson@cmu.edu}{larsson@cmu.edu}} \and Jonghwa Park\thanks{ETH Zurich, Department of Mathematics, \href{mailto:jonghwa.park@math.ethz.ch}{jonghwa.park@math.ethz.ch}} \and Johannes Wiesel\thanks{University of Copenhagen, Department of Mathematics, \href{mailto:wiesel@math.ku.dk}{wiesel@math.ku.dk}} }
\date{\today}

\maketitle

\begin{abstract}
The Wasserstein distance $\mathcal{W}_p$ is an important instance of an optimal transport cost. Its numerous mathematical properties as well as applications to various fields such as mathematical finance and statistics have been well studied in recent years. The adapted Wasserstein distance $\mathcal{A}\mathcal{W}_p$ extends this theory to laws of discrete time stochastic processes in their natural filtrations, making it particularly well suited for analyzing time-dependent stochastic optimization problems.
    
    While the topological differences between $\mathcal{A}\mathcal{W}_p$ and $\mathcal{W}_p$ are well understood, their differences as metrics remain largely unexplored beyond the trivial bound $\mathcal{W}_p\lesssim \mathcal{A}\mathcal{W}_p$. This paper closes this gap by providing upper bounds of $\mathcal{A}\mathcal{W}_p$ in terms of $\mathcal{W}_p$ through investigation of the smooth adapted Wasserstein distance. Our upper bounds are explicit and are given by a sum of $\mathcal{W}_p$, Eder's modulus of continuity and a term characterizing the tail behavior of measures. As a consequence, upper bounds on $\mathcal{W}_p$ automatically hold for $\mathcal{AW}_p$ under mild regularity assumptions on the measures considered. A particular instance of our findings is the inequality $\mathcal{A}\mathcal{W}_1\le C\sqrt{\mathcal{W}_1}$ on the set of measures that have Lipschitz kernels. 

    Our work also reveals how smoothing of measures affects the adapted weak topology. In fact, we find that the topology induced by the smooth adapted Wasserstein distance exhibits a non-trivial interpolation property, which we characterize explicitly: it lies in between the adapted weak topology and the weak topology, and the inclusion is governed by the decay of the smoothing parameter.

    \medskip
    
    \noindent\emph{Keywords:} (adapted) Wassestein distance, optimal transport, modulus of continuity, relative compactness
\end{abstract}

\begin{section}{Introduction}\label{sec:intro}

For $N\in \N$ we denote by $\sp_p(\R^N)$ the set of all Borel probability measures on $\R^N$ that have finite $p$-moments, where $1\le p<\infty$ is fixed throughout this paper. The Wasserstein distance, defined via 
\begin{align}\label{eq:wass}
    \w_p(\mu, \nu)
    :=\left(\inf_{\pi \in \cpl(\mu, \nu)}\int_{\R^N\times \R^N} \abs{x-y}^p \pi(dx, dy)\right)^{1/p},
\end{align}
is a metric on $\sp_p(\R^N)$ that metrizes weak convergence plus convergence of $p$-moments. Here $\abs{\cdot}$ denotes the Euclidean norm on $\R^N$ and $\cpl(\mu, \nu)$ is the set of all couplings between $\mu\in\sp_p(\R^N)$ and $\nu\in \sp_p(\R^N)$, i.e., those probability measures on $\R^N\times \R^N$ whose first marginal is $\mu$ and whose second marginal is $\nu$. Computing the Wasserstein distance is an instance of the Kantorovich optimal transport problem \cite{kantorovich1942translocation}, and its numerous mathematical properties are well studied. We refer to \cite{villani2009optimal, villani2021topics, santambrogio2015optimal} for a general overview.

While the Wasserstein distance between laws on $\R^N$ has seen a surge of applications in the last few years, the situation changes  when $\mu,\nu$ are interpreted as laws of discrete-time stochastic processes in their natural filtrations. In this case, measuring the distance between $\mu$ and $\nu$ via $\w_p(\mu, \nu)$ is often inadequate when considering time-dependent optimization problems such as optimal stopping problems and multistage optimization. The key reason for this is that the formulation \eqref{eq:wass} does not take the time structure of the laws $\mu,\nu$ into account when optimizing over all possible couplings on $\R^N$. The restriction to so-called \emph{bicausal} or \emph{adapted} couplings described in the next paragraph addresses this issue.

Let us first  set up some notation: consider two stochastic processes $X=(X_1, \ldots, X_T)$ and $Y=(Y_1,\ldots, Y_T)$ in finite discrete time, where both $X_t$ and $Y_t$ are $\R^d$-valued random variables for $t\in \{1,2,\ldots, T\}$. Throughout the paper, $T\in \N$ denotes the number of time steps and $d\in \N$ is the dimension of the state space. Consider now a Monge map $S=(S_1, \ldots, S_T):(\R^d)^T\to (\R^d)^T$ from the law of $X$ to the law of $Y$, which satisfies
\begin{align}
    Y_t\stackrel{d}{=} S_t(X_1, \ldots, X_T) \quad \text{ for } t\in \{1,2,\ldots, T\}.
\end{align}
Without further restrictions on the map $S$, $Y_t$ can depend on the whole stochastic process $X$. In other words, it is impossible to  determine $Y_t$ only from the information $(X_1, \dots, X_t)$ as the map $S_t$ uses information about the evolution of $X$ \emph{after} time $t$ in general. In order to respect the filtration generated by $X$, it is reasonable to require $S$ to be adapted, i.e., $S_t$ should only be a function of the first $t$ coordinates. As shown e.g. in \cite{beiglbock2018denseness, beiglbock2022denseness}, the notion of adapted mappings is generalized to bicausal couplings as follows.

\begin{df}
Let $\mu$ and $\nu$ be two probability measures on $(\R^d)^T$. A coupling $\pi\in \cpl(\mu, \nu)$ is called \textit{bicausal} if for $(X,Y)\sim \pi$ and $t\in \{1,2,\ldots, T-1\}$,
\begin{align}
    (Y_1, \ldots, Y_t) \text{ and } (X_{t+1}, \ldots, X_T) \text{ are conditionally independent given } X_1, \ldots, X_t
\end{align}
and
\begin{align}
    (X_1, \ldots, X_t) \text{ and } (Y_{t+1}, \ldots, Y_T) \text{ are conditionally independent given } Y_1, \ldots, Y_t.
\end{align}
The set of all bicausal couplings between $\mu$ and $\nu$ is denoted by $\cplbc(\mu, \nu)$. 
\end{df}

Heuristically speaking, $\pi \in \cplbc(\mu, \nu)$ couples two processes $X\sim \mu$ and $Y\sim \nu$ such that predicting $Y$ at time $t$ requires at most the information of $X$ up to time $t$ and vice versa. We define the adapted Wasserstein distance as the optimal transport cost over bicausal couplings.
\begin{df}[The adapted Wassertein distance]\label{df:aw}
    Let $\mu, \nu\in \sp_p((\R^d)^T)$. The adapted Wasserstein distance between $\mu$ and $\nu$ is given by
    \begin{align}
        \a\w_p(\mu, \nu)
        :=\left(\inf_{\pi\in \cplbc(\mu, \nu)}\int_{(\R^d)^T\times (\R^d)^T} \sum_{t=1}^{T} \abs{x_t-y_t}^p \pi(dx, dy)\right)^{1/p}.
    \end{align}
\end{df}

Much of the previous work has been on characterizing the topology generated by $\a\w_p$; we provide a literature review of these results in Section \ref{sec:lit}. Here we only remark that by Definition \ref{df:aw} one can easily derive the inequality $\w_p\le C\a\w_p$ (for some constant $C>0$ depending only on $p, T$). As a consequence, the adapted Wasserstein topology is finer than the Wasserstein topology. The following example going back at least to \cite{backhoff2020adapted} shows that this inclusion is strict in general. We refer to it as the standard example and use it to illustrate our results throughout this paper.

\begin{ex}[Standard example]\label{ex:standardex}
    Let $\mu_{\ep}=\frac{1}{2}\delta_{(\ep, 1)}+\frac{1}{2}\delta_{(-\ep, -1)}$ and $\mu=\frac{1}{2}\delta_{(0,1)}+\frac{1}{2}\delta_{(0,-1)}$. As $\ep\to 0$ we have $\w_p(\mu, \mu_{\ep})=\ep\to 0$, but $\a\w_p(\mu, \mu_{\ep})=(\ep^p+2^{p-1})^{1/p}\nrightarrow 0$. 
\end{ex}

\begin{figure}[ht]
    \centering
    \includegraphics[width = 14cm, height = 4cm]{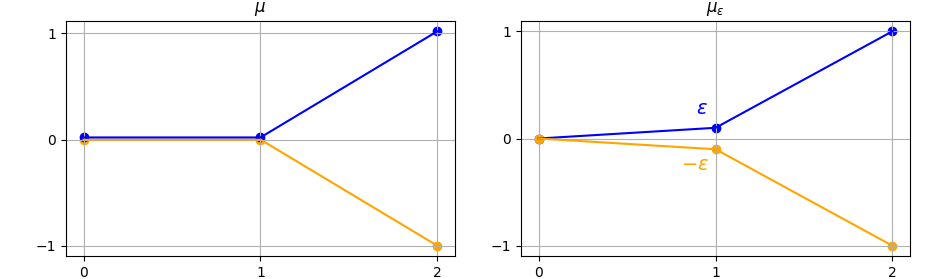}
    \caption{Paths of $\mu=\frac{1}{2}\delta_{(0,1)}+\frac{1}{2}\delta_{(0,-1)}$ and $\mu_{\ep}=\frac{1}{2}\delta_{(\ep, 1)}+\frac{1}{2}\delta_{(-\ep, -1)}$.} 
    \label{fig:example_path}
\end{figure}

In fact, the adapted Wasserstein topology is the coarsest topology that makes optimal stopping problems continuous \cite{backhoff2020all}. \cite{pflug2012distance, pflug2014multistage, glanzer2019incorporating, backhoff2020adapted, acciaio2020causal} demonstrate, that time-dependent optimization problems from mathematical finance such as utility maximization, pricing and hedging derivatives, risk measurements, etc., are actually Lipschitz continuous with respect to $\a\w_p$.

While the differences between the Wasserstein topology and the adapted Wasserstein topology are thus well understood, the comparison between the \emph{metrics} $\w_p$ and $\a\w_p$ is still largely open to the best of our knowledge. This article aims to fill this gap. We want to answer the following question:

\begin{tcolorbox}
Can we identify a class of subsets $K\subseteq \sp_p((\R^d)^T)$, on which $\a\w_p$ and $\w_p$ are nearly equivalent, i.e. there are ``nice" explicit upper bounds of $\a\w_p$ in terms of $\w_p$?
\end{tcolorbox}

Except for trivial bounds, the above is an ill-posed task for $K=\sp_p((\R^d)^T)$. In fact, Example~\ref{ex:standardex} clearly shows that $\a\w_p$ and $\w_p$ cannot be equivalent in any meaningful sense. However, this example is arguably tailor made to showcase the differences between $\w_p$ and $\a\w_p$. In fact, $\{\mu_{\ep}:\ep>0\}$ are chosen in such a way that the conditional distributions $x_1\mapsto \P(X^{\ep}_2\in \,\cdot\,|X^{\ep}_1=x_1)$ become increasingly irregular for $(X^{\ep}_1, X^{\ep}_2)\sim \mu_{\ep}$: recall that we have
\begin{align}
\P(X^{\ep}_2= 1|X^{\ep}_1=\epsilon)=1\quad \text{and} \quad \P(X^{\ep}_2=-1|X^{\ep}_1=-\epsilon)=1.
\end{align}
In other words, $\{\mu_{\ep}: \ep>0\}$ does \emph{not} have equicontinuous kernels. In many situations it is however reasonable to restrict to subsets  $K\subseteq \sp_p((\R^d)^T)$, whose elements have reasonably smooth kernels. It turns out that such subsets $K$ are exactly the $\a\w_p$-relatively compact sets.

\begin{prop}[Eder \cite{eder2019compactness}]\label{prop:awrelcpt}
A set $K\subseteq \sp_p((\R^d)^T)$ is $\a\w_p$-relatively compact if and only if $K$ is $\w_p$-relatively compact and
\begin{align}
\lim_{\delta\to 0}\sup_{\mu\in K}\omega^{t,p}_{\mu}(\delta)=0
\text{ for all } t\in \{1,2,\ldots, T-1\},
\end{align}
where $\omega^{t,p}_{\mu}:(0,\infty)\to (0,\infty)$ is \emph{the $p$-modulus of continuity of $\mu$} defined via
\begin{align}
\omega_{\mu}^{t,p}(\delta)
            :=\sup
            \left\{\left(\E[\w_p\left(\P(X_{t+1}\in \,\cdot\,\st X_1, \ldots, X_t), \P(Y_{t+1}\in \,\cdot\,\st Y_1, \ldots, Y_t)\right)^p]\right)^{1/p} \right\}.
\end{align}
Here, the supremum is taken over all $X,Y\sim \mu$ such that $\norm{(X_1, \ldots, X_t)-(Y_1, \ldots, Y_t)}_{L^p}<\delta$.
\end{prop}

Let us emphasize here that focusing on the $\a\w_p$-relatively compact subsets of $\sp_p((\R^d)^T)$ is usually not a big restriction in practical applications, and sometimes even arises naturally. We refer to Section \ref{sec:awtopology} for a more detailed discussion.

Proposition \ref{prop:awrelcpt} will prove central in this article, as we will essentially show that the inequality $\w_p\le C\a\w_p$ can be reversed exactly on the $\a\w_p$-relatively compact sets. To get some intuition for this, assume for now that $p=1$ and $K\subseteq \sp_{1}((\R^d)^T)$ is $\a\w_1$-compact. Then it essentially follows from \cite[Lemma 7]{pammer2024note}, that the Wasserstein topology and the adapted Wasserstein topology agree on $K$. As a consequence, for any $\sigma>0$ there exists a constant $C$ depending on $\sigma$ and $K$ such that we have
\begin{align}\label{eq:gudi}
\a\w_1(\mu,\nu)\le C\w_1(\mu,\nu) +\sigma
\end{align}
for all $\mu, \nu\in K$. While theoretically appealing, this result is of very limited practical relevance as long as one cannot characterize the dependence of $C$ on $\sigma$ and $K$. One of the main contributions of this article is to make the bound~\eqref{eq:gudi} explicit. In fact, we show that (ignoring dimensional constants that can be made explicit) for any positive real numbers $0<\sigma<R$, 
\begin{align}
\a\w_1(\mu,\nu)
\lesssim \frac{R}{\sigma}\w_1(\mu,\nu) + h(\sigma,K) + g(R,K),\label{eq:gudi_goal}
\end{align}
where the function $h$ is derived from Eder's modulus of continuity $\omega$ (see Proposition \ref{prop:awrelcpt}) and the function $g$ characterizes the tail behavior of measures in $K$ outside a Euclidean ball of radius $R$. In particular, $\lim_{\sigma\to 0} h(\sigma,K)=0$ and $\lim_{R\to \infty } g(R,K) =0$. We refer to Corollary \ref{cor:boundingAW} for an exact statement of this bound.

In other words, \eqref{eq:gudi_goal} implies that $\a\w_1$ can be bounded by $\w_1$ up to an error that is uniformly bounded on $\a\w_1$-relatively compact sets. This seemingly simple result shows its true power in applications: the time-dependent optimization problems (e.g. utility maximization, pricing, hedging) mentioned above are not just continuous in $\a\w_1$, but are actually \emph{continuous in $\w_1$} on $\a\w_1$-relatively compact sets and \emph{estimates are explicitly computable}. This significantly simplifies the analysis of time-dependent optimization problems. Indeed, it is actually not necessary to design new estimates specifically tailored for $\a\w_1$ as e.g. done in \cite{backhoff2022estimating}; instead one can simply rely on the rich theory that has already been developed for $\w_1$ in the last decades. We give some examples illustrating this new methodology below, showing how one can derive finite-sample guarantees for the adapted Wasserstein distance. We believe that  \eqref{eq:gudi_goal} has much further-reaching implications for the theory of adapted transports however, which we aim to investigate in future work.

\subsection{The smooth adapted Wasserstein distance}

In order to derive \eqref{eq:gudi_goal}, we first need to introduce a variant  of $\a\w_p$: 

\begin{df}
Given a noise distribution $\xi$ and a bandwidth parameter $\sigma>0$, we define \emph{the smooth adapted Wasserstein distance} via $$\a\w^{(\sigma)}_p(\mu, \nu):=\a\w_p(\mu\ast \xi_{\sigma}, \nu\ast \xi_{\sigma}).$$
Here $\ast$ denotes the convolution of two measures and $\xi_{\sigma}$ is the law of $(x\mapsto \sigma x)_{\#}\xi$, i.e. if $Z\sim \xi$, then $\sigma Z\sim \xi_{\sigma}$.
\end{df}

 The non-adapted counterpart of $\a\w^{(\sigma)}_p$ is called the smooth Wasserstein distance $$\w^{(\sigma)}_p(\mu, \nu):=\w_p(\mu\ast \xi_{\sigma}, \nu\ast\xi_{\sigma})$$ and is by now well studied in the statistics and machine learning literature, mainly for the purpose of relaxing the curse of dimensionality. Indeed, denoting the empirical measure of $\mu$ by $\boldsymbol{\mu_n}:=\frac{1}{n}\sum_{j=1}^n\delta_{X^{(j)}}$ where $X^{(1)}, X^{(2)}, \ldots$ are i.i.d samples of $\mu$, it turns out that $\sqrt{n}\w^{(\sigma)}_p(\mu, \boldsymbol{\mu_n})$ has a weak limit under suitable moment assumptions. We refer to \cite{goldfeld2020convergence, goldfeld2020gaussian, goldfeld2020asymptotic, sadhu2021limit, nietert2021smooth, goldfeld2024limit, goldfeld2024statistical} and the references therein for this line of research. On the other hand, $\a\w^{(\sigma)}_p$ is quite new to the best of our knowledge, and has only been studied very recently in the literature, cf. \cite{blanchet2024empirical, hou2024convergence}. These results suggest that, as in the non-adapted case, the empirical convergence rate only exhibits a mild dependence on the dimension after smoothing.

While the statistical properties of $\a\w^{(\sigma)}_p(\boldsymbol{\mu_n}, \mu)$ are thus at least partially understood, the arguably more fundamental question of how smoothing affects the adapted Wasserstein topology is still wide open to the best of our knowledge. Recall that it is well known that $\w_p$ and $\w^{(\sigma)}_p$ generate the same topology (see \cite{nietert2021smooth}), and it is thus natural to conjecture that the same holds for $\a\w_p$ and $\a\w^{(\sigma)}_p$. Perhaps rather surprisingly, we show that this conjecture is false in general. In fact we prove that, essentially,
\begin{align}
    \text{Wasserstein topology}
    \,\subseteq\,
    \text{smooth adapted Wasserstein topology}
    \,\subseteq\,
    \text{adapted Wasserstein topology},
\end{align}
and the inclusions in the above display are equalities under suitable assumptions on \textit{the decay of $\sigma$} (when thinking about sequential convergence). We return to this discussion in Section \ref{sec:intro_topology}.

\subsection{\texorpdfstring{Bounding $\a\w_1$ by $\w_1$}{TEXT}} \label{sec:intro_bound}

We are now in a position to give  a heuristic derivation of \eqref{eq:gudi_goal}. To keep the notation simple, we take $p=1$, $T=2$ and refer to Corollary \ref{cor:boundingAW} for the general statement. Assuming $K\subseteq \sp_1((\R^d)^2)$ is $\a\w_1$-relatively compact, an application of the triangle inequality yields
\begin{align}
    \a\w_1(\mu, \nu)
    \le \a\w^{(\sigma)}_1(\mu, \nu)
    +2\sup_{\mu\in K}\a\w_1(\mu, \mu\ast \xi_{\sigma})
    \label{eq:gudi_tri}
\end{align}
for any $\mu, \nu\in K$ and $\sigma>0$. In order to bound $\a\w_1$, it is thus sufficient to bound the two terms on the right-hand side of \eqref{eq:gudi_tri} separately. These bounds constitute the two main results of our paper and we refer the reader to Section \ref{sec:main results} for exact statements. We start with a bound on the smooth adapted Wasserstein distance.

\begin{thm}[Bounding $\a\w^{(\sigma)}_1$ by $\w_1$; special case of Theorem \ref{thm:awsigmaw1}] \label{thm_intro1}
For any $\sigma, R>0$ we have
\begin{align}
    \a\w^{(\sigma)}_1(\mu, \nu)
    \le C\left( \frac{R}{\sigma}\w_1(\mu, \nu)
    +\sup_{\mu\in K}\int_{\{\abs{x}\ge R\}}\abs{x} (\mu\ast\xi_{\sigma})(dx)\right)\label{eq:gudi_bdaw}
\end{align}
for some $C>0$ that depends only on a noise distribution $\xi$.
\end{thm}
In other words, $\a\w^{(\sigma)}_1$ can be bounded by $\w_1$ up to an error term depending only on the tails of $\mu\ast\xi_{\sigma}$. In the next section we show that the scaling of order $\w_1/\sigma$ in the upper bound is essentially optimal. Our second main result pertains to the question, how much information is lost by considering $\mu\ast \xi_{\sigma}$ instead of its unperturbed version $\mu$. 

\begin{thm}[Bandwidth effect of $\sigma$; special case of Theorem \ref{thm:muandmusigma}]\label{thm_intro2}
We have
\begin{align}
    \sup_{\mu\in K}\a\w_1(\mu, \mu\ast \xi_{\sigma})
    \le C\left( \sigma+\sup_{\mu\in K}\omega^{1,1}_{\mu}(\sigma)\right)\label{eq:gudi_sigma}
\end{align}
for some $C>0$ that depends only on a noise distribution $\xi$.
\end{thm}

In conclusion, the difference between $\mu$ and its smoothed version $\mu\ast \xi_{\sigma}$ is bounded by Eder's modulus of continuity $\omega^{1,1}_{\mu}$ uniformly on the $\a\w_1$-relatively compact sets.

Combining Theorems \ref{thm_intro1} and \ref{thm_intro2} with  \eqref{eq:gudi_tri} finally yields the following result:

\begin{cor}[Bounding $\a\w_1$ by $\w_1$; special case of Corollary \ref{cor:boundingAW}]
For all $0<\sigma<R$ we have
\begin{align}
    \a\w_1(\mu, \nu)
    \le C\left( \frac{R}{\sigma}\w_1(\mu, \nu)
    +\left(\sigma+\sup_{\mu\in K}\omega^{1,1}_{\mu}(\sigma)\right)
    +\sup_{\mu\in K}\int_{\{\abs{x}\ge R\}}\abs{x}\mu(dx)\right)
\end{align}
for some $C>0$ that depends only on $d$.
\end{cor}

Returning to \eqref{eq:gudi_goal}, we conclude that we can choose
\begin{align}
    h(\sigma, K)
    =\sigma+\sup_{\mu\in K}\omega^{1,1}_{\mu}(\sigma),\quad
    g(R, K)
    =\sup_{\mu\in K}\int_{\{\abs{x}\ge R\}}\abs{x}\mu(dx).
\end{align}

To summarize, the functions $h,g$ are remarkably explicit and exactly mirror Eder's characterization of $\a\w_1$-relatively compact sets as stated in Proposition \ref{prop:awrelcpt}. In particular, $g$ controls $\w_1$-relative compactness of $K$, while $h$ controls the regularity of the kernels in $K$. 
The proof of the general result stated in Corollary \ref{cor:boundingAW} essentially proceeds in the same fashion as described above, iterating over time periods $t=1, \dots, T-1$.

\subsection{Comparison of Wasserstein topologies }\label{sec:intro_topology} 

To the best of our knowledge, general properties of the topology generated by $\a\w_p^{(\sigma)}$ have not been studied. Given our explicit estimates in Theorems \ref{thm_intro1} and \ref{thm_intro2}, we can now address this question in great generality. To make our results concise, we consider a sequence $(\mu_n)_{n\in \N}$ in $\sp_p((\R^d)^T)$ and an additional measure $\mu\in \sp_p((\R^d)^T)$. We  refer to Section \ref{sec:topology} for exact statements.

Perhaps surprisingly, for fixed $\sigma>0$, the topology generated by $\a\w^{(\sigma)}_p$ is in fact the (vanilla) Wasserstein topology, i.e. 
\begin{align*}
\lim_{n\to \infty} \a\w^{(\sigma)}_p(\mu,\mu_n) =0 \quad 
\Leftrightarrow \quad \lim_{n\to \infty} \w^{(\sigma)}_p(\mu,\mu_n) =0 \quad
\Leftrightarrow \quad \lim_{n\to \infty} \w_p(\mu,\mu_n) =0.
\end{align*}
In this sense, all adapted topologies are equal as famously stated in \cite{backhoff2020adapted}, \emph{but only if there is no uncertainty about the observed distributions $\mu, \mu_n$}. As soon as these distributions are perturbed by some (arbitrarily small) independent noise $\xi_\sigma$, there is virtually no difference between the Wasserstein topology and the adapted Wasserstein topology. Phrased differently, in this case the finer structure of the adapted Wasserstein topology is absorbed by the noise $\xi_\sigma$.

The picture is different when one considers variable noise $(\xi_{\sigma_n})_{n\in \N}$ instead. Now there is an interesting tradeoff between the size of the noise $\xi_{\sigma_n}$ and the convergence speed of $\w_1(\mu,\mu_n)$ on the one hand (cf. Theorem~\ref{thm_intro1}) and the regularity of kernels on the other hand (cf. Theorem~\ref{thm_intro2}). Figure~\ref{fig:implication} illustrates this tradeoff in full generality. Here, the function $h^{t,p}_{\mu_n}$ is an \textit{iterated modulus of continuity}, e.g., $h^{1,p}_{\mu_n}(\sigma_n)=\omega^{1,p}_{\mu_n}(\sigma_n)$, $h^{2,p}_{\mu_n}(\sigma_n)=\omega^{2,p}_{\mu_n}(\sigma_n+h^{1,p}_{\mu_n}(\sigma_n))$ and $h^{3,p}_{\mu_n}(\sigma_n)=\omega^{3,p}_{\mu_n}(\sigma_n+h^{1,p}_{\mu_n}(\sigma_n)+h^{2,p}_{\mu_n}(\sigma_n))$. See Theorem~\ref{thm:muandmusigma} for the precise definition of $h^{t,p}_{\mu_n}$.

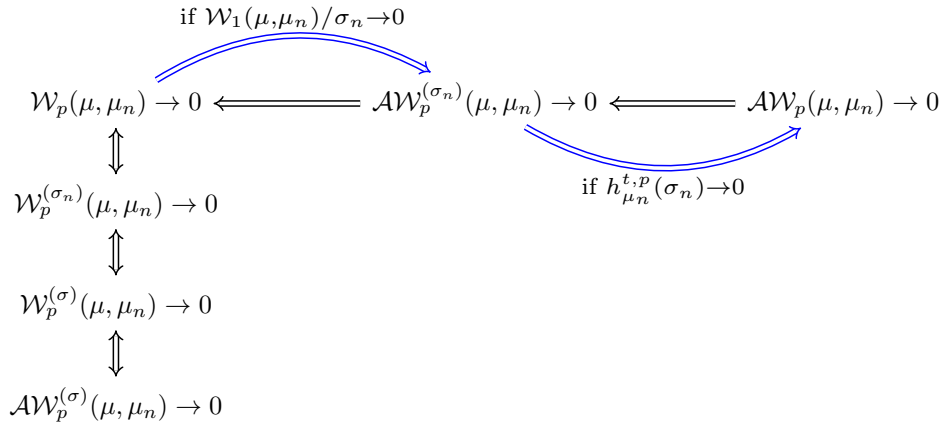
\begin{figure}[ht]
    \centering
    \begin{tikzcd}[arrows={line width=0.2mm}]
    &&\w_p(\mu, \mu_n)\to 0\arrow[Leftarrow]{rr} \arrow[blue, Rightarrow, bend left]{rr}[black,swap, above]{\scalebox{1.25}{$\text{if } \w_1(\mu, \mu_n)/\sigma_n\to 0$}}
    && \a\w^{(\sigma_n)}_p(\mu, \mu_n)\to 0 \arrow[Leftarrow]{rr}\arrow[blue, Rightarrow, bend right]{rr}[black,swap]{\scalebox{1.25}{$\text{if } h^{t,p}_{\mu_n}(\sigma_n)\to 0$}}  
    && \a\w_p(\mu, \mu_n)\to 0\\
    &&\w^{(\sigma_n)}_p(\mu, \mu_n)\to 0 \arrow[Leftrightarrow]{u}\\
    &&\w^{(\sigma)}_p(\mu, \mu_n)\to 0 \arrow[Leftrightarrow]{u}\\
    &&\a\w^{(\sigma)}_p(\mu, \mu_n)\to 0 \arrow[Leftrightarrow]{u}
    \end{tikzcd}
    \caption{Comparison of topologies.}
    \label{fig:implication}
\end{figure}

In the sense of Figure \ref{fig:implication}, convergence in $\a\w^{(\sigma_n)}_p$ is in general neither the same as convergence in $\w_p$ nor convergence in $\a\w_p$ and the smoothed adapted Wasserstein distance $\a\w^{(\sigma_n)}_p$ exhibits a non-trivial \emph{interpolation property} between the different topologies. 
To motivate the above results, it is perhaps most instructive to return to the standard Example \ref{ex:standardex}, where we can compute all relevant quantities explicitly. By convolving $\mu_{\ep_n}=\frac{1}{2}\delta_{(\ep_n, 1)}+\frac{1}{2}\delta_{(-\ep_n, -1)}$ with Gaussian noise $\xi_{\sigma_n}=\n(0, \sigma^2_n\Id_{2})$, we see from Figure \ref{fig:2dfigure} that kernels of $\mu_{\ep_n}$ are smoothed out. As $\ep_n\to 0$, we can show that
\begin{anumerate}
    \item $\frac{\ep_n}{\sigma_n}\to 0$ if and only if $\frac{\w_1(\mu,\mu_{\ep_n})}{\sigma_n}\to 0$ if and only if $\a\w^{(\sigma_n)}_p( \mu, \mu_{\ep_n})\to 0$, 
    \item $\frac{\ep_n}{\sigma_n}\to \infty$ if and only if $\omega^{1,p}_{\mu_{\ep_n}}(\sigma_n)\to 0$ if and only if $\a\w_p(\mu_{\ep_n}, \mu_{\ep_n}\ast \xi_{\sigma_n})\to 0$;
\end{anumerate}
see Appendix \ref{appendix:pfex} for details. Roughly speaking, if $\sigma_n$ decays slowly enough (so that $\ep_n/\sigma_n\to 0$), the behavior of $\a\w^{(\sigma_n)}_p(\mu, \mu_{\ep_n})$ mirrors the one of $\w_p(\mu, \mu_{\ep_n})$. Similarly, if $\sigma_n$ decays fast enough (so that $\ep_n/\sigma_n\to \infty$), we can use the triangle inequality to see that $\a\w^{(\sigma_n)}_p(\mu, \mu_{\ep_n})$ and $\a\w_p(\mu, \mu_{\ep_n})$ converge toward the same value. In this sense, the tradeoff illustrated in Figure \ref{fig:implication} captures exactly the right scaling of $\w_p(\mu, \mu_{\ep_n})$ and $\sigma_n$, at least for the standard example.

\begin{figure}[ht]
\centering
\begin{subfigure}{0.48\textwidth}
    \includegraphics[width=\textwidth]{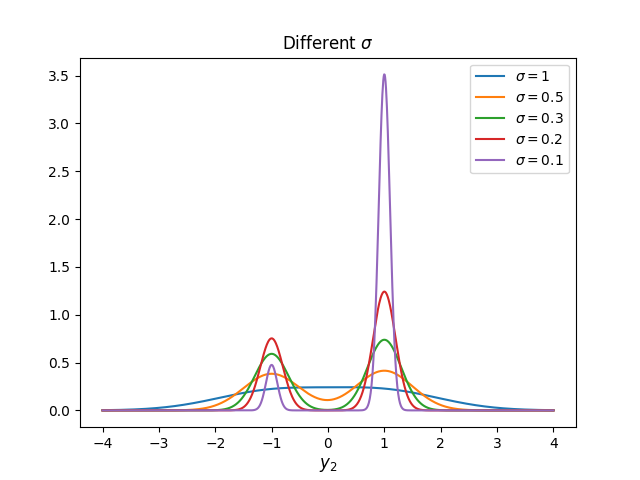}
    \caption{$\ep=0.1$.}
    \label{fig:diffsig}
\end{subfigure}
\hfill
\begin{subfigure}{0.48\textwidth}
    \includegraphics[width=\textwidth]{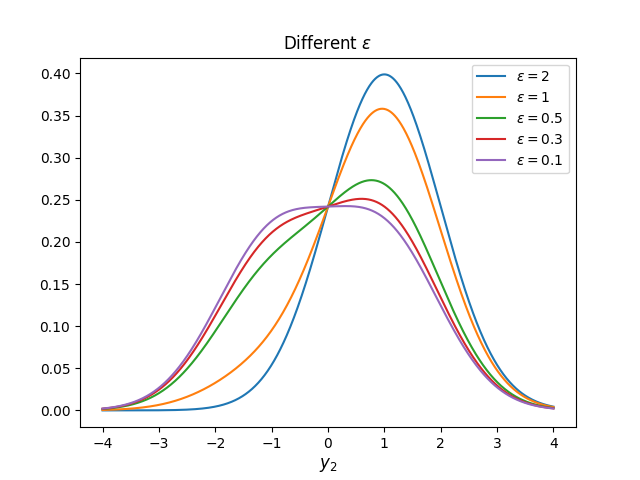}
    \caption{$\sigma=1$.}
    \label{fig:diffep}
\end{subfigure}
        
\caption{Density functions of measures $\P(Y_2\in dy_2 \st Y_1=\ep)$ for $(Y_1, Y_2)\sim \mu^{\sigma}_{\ep}$.}
\label{fig:2dfigure}
\end{figure}


\subsection{Related work} \label{sec:lit}

Adapted variants of the Wasserstein distance have been studied independently by many authors. In \cite{Al81}, Aldous introduced the notion of \textit{extended weak convergence}, which is essentially induced by the joint law between the original process and its corresponding prediction process. Further derivations of adapted topologies can be found in \cite{hoover1984adapted, ruschendorf1985wasserstein, hellwig1996sequential, lassalle2018causal, acciaio2019extended, bonnier2023adapted} and the references therein. Remarkably, \cite{backhoff2020all} shows that all these seemingly different approaches define the same topology, which we refer to as the adapted Wasserstein topology throughout this paper. Based on the \textit{information topology} proposed by \cite{hellwig1996sequential}, Eder  characterizes the relatively compact sets in the adapted Wasserstein topology in \cite{eder2019compactness}. The modulus of continuity proposed by Eder plays a central role in our work.

As discussed earlier, adapted transport maps induce bicausal couplings and as for their non-adapted counterparts, this inclusion is dense under mild regularity assumptions on the kernels \cite{beiglbock2018denseness, beiglbock2022denseness}. The key idea is to embed bicausal couplings into enlarged spaces and approximate them by adapted maps. In Lemma \ref{lem:epcoupling} we adapt these results to approximate a bicausal coupling between $\mu$ and $\mu\ast \xi_{\sigma}$ by transport maps in order to bound the adapted Wasserstein distance between them. 

We also mention that duality of the adapted Wasserstein distance is investigated in \cite{backhoff2017causal, eckstein2024computational}. Furthermore, \cite{backhoff2017causal} introduce a dynamic programming principle for the adapted Wasserstein distance, which we will use throughout this paper.

Notably, Eckstein--Pammer show in \cite{eckstein2024computational}, that the total variation distance and the \textit{adapted total variation distance} $\av(\mu, \nu):=\inf_{\pi\in \cplbc(\mu, \nu)}\pi(\{x\neq y\})$ are equivalent as metrics. To the best of our knowledge, this is the first result that gives upper and lower bounds of an adapted optimal transport cost by an optimal transport cost. Our work uses this estimate to bound $\a\w_p$ in terms of $\w_p$.

The recent preprint \cite{acciaio2025estimating} investigates the upper bound \eqref{eq:gudi_goal} under the assumption that the densities of $\mu$ and $\nu$ are elements of a Sobolev space. Their work is motivated by a preliminary version of the present paper and was made publicly available after the submission of this work. One of their contributions is an improvement of the bound obtained by Eckstein--Pammer \cite{eckstein2024computational}. In the original work \cite{eckstein2024computational}, the bounding constants depend exponentially on the number of time steps $T$, whereas \cite{acciaio2025estimating} shows that the constants grow linearly in $T$. 

Regarding applications, \cite{pflug2012distance, pflug2014multistage, glanzer2019incorporating} show that many multistage optimization problems such as computing bid prices, are H\"older continuous with respect to $\a\w_p$. \cite{backhoff2020adapted, acciaio2020causal} extend this idea to a continuous time setting in mathematical finance and show that utility maximization, risk measurement and pricing/hedging problems are Lipschitz continuous with respect to $\a\w_1$.

Smoothing measures by convolving with Gaussian noise $\xi$ is a standard tool in statistics, and is often used to obtain dimension-free empirical approximations. In optimal transport, the smooth Wasserstein distance and its concentration has been studied extensively.
We refer e.g. to \cite{nietert2021smooth, goldfeld2020convergence, goldfeld2020gaussian} for both the slow rate $\E[\w^{(\sigma)}_p(\mu, \boldsymbol{\mu_n})^p]\le C/\sqrt{n}$ and the fast rate $\E[\w^{(\sigma)}_p(\mu, \boldsymbol{\mu_n})]\le C/\sqrt{n}$ under suitable moment conditions. Distributional limits of $\sqrt{n}\w^{(\sigma)}_p(\mu, \boldsymbol{\mu_n})$ can be found in \cite{goldfeld2020asymptotic, sadhu2021limit} for the case $p=1$ and in \cite{goldfeld2024limit, goldfeld2024statistical} for $p>1$.

It is well-known that the usual empirical measure $\boldsymbol{\mu_n}$ is not adequate for empirical approximations in the adapted Wasserstein topology, i.e., $\lim_{n\to \infty} \a\w_p(\mu,\boldsymbol{\mu_n})\neq 0$ in general. To alleviate this, \cite{backhoff2022estimating} considers so-called \textit{adapted empirical measures} for measures supported on compact sets, and prove the same dimension-dependent convergence rates as are known for their non-adapted counterparts. \cite{acciaio2022convergence} extends this result to general measures. 

Quite recently, there have been  first efforts to study properties of the smooth adapted Wasserstein distance $\a\w^{(\sigma)}_p$. Motivated by a martingale pair test, \cite{blanchet2024empirical} obtain the  dimension-free convergence rate $1/\sqrt{n}$ for the so-called smoothed martingale projection distance. Under suitable moment assumptions, \cite{hou2024convergence} establishes rates for $\E[\a\w^{(\sigma)}_1(\mu, \boldsymbol{\mu_n})]$.

\begin{subsection}{Organization of the paper}
The rest of the paper is organized as follows. We end this introduction by setting up notation in Section~\ref{sec:intro_nota}.  In Section~\ref{sec:background} we present essential background materials for this paper, including the characterization of $\a\w_p$-relatively compact sets by \cite{eder2019compactness} and basic properties of smooth Wasserstein distances. In Section~\ref{sec:main results}, we provide the exact statements of the results that were briefly mentioned in Section \ref{sec:intro_bound} and Section \ref{sec:intro_topology}. Proofs of results in Section~\ref{sec:background} are contained in Section~\ref{sec:pf_background}. Section~\ref{sec:pf_main results} is devoted to proofs of results in Section~\ref{sec:main results}. In Appendix~\ref{appendix:pfex}, the convergence results regarding Example \ref{ex:standardex} are explained in detail. Appendix~\ref{appendix:monge} contains a denseness result used in the proof of Lemma \ref{lem:epcoupling}.
\end{subsection}

\begin{subsection}{Notation}\label{sec:intro_nota}
As mentioned above, $\abs{\cdot}$ is the Euclidean norm, and we denote the scalar (dot) product by $\cdot$. Throughout the paper,  $T\in \N$ is the number of time steps and $d\in \N$ is the dimension of the state space. The set of all Borel probability measures on $(\R^d)^T$  is denoted by $\sp((\R^d)^T)$. For $1\le p<\infty$, $\sp_p((\R^d)^T)$ is the set of all $\mu\in \sp((\R^d)^T)$ that have finite $p$-moments. If $\mu\in \sp((\R^d)^T)$, then we write $M_p(\mu):=\int_{(\R^d)^T}\abs{x}^p \mu(dx)$ for the $p$-moment of $\mu$. We denote the pushforward measure of $\mu$ under a Borel function $g$ by $g_{\#}\mu$, i.e. $g_{\#}\mu(A)=\mu(\{x:\, g(x)\in A\})$ for all Borel sets $A$.

For $x\in (\R^d)^T$, we use the shorthand notation $x_t$ to denote the $t$-th coordinate of $x$ and $x_{1:t}:=(x_1, \ldots, x_t)$ for $t\in \{1,2,\ldots, T\}$. In particular, $x=x_{1:T}$. Analogously, for $\mu\in \sp((\R^d)^T)$ and $t\in \{1,2,\ldots, T\}$, the measure $\mu_{t}\in \sp(\R^d)$ is the projection of $\mu$ onto the $t$-th coordinate, i.e., if $P_t:(\R^d)^T\to \R^d$ is defined as $P_t(x_1, \ldots, x_T)=x_t$, then $\mu_t=(P_t)_{\#}\mu$. Also, $\mu_{1:t}\in \sp((\R^d)^t)$ is the pushforward measure of $\mu$ through the projection map $P_{1:t}$ onto the first $t$ coordinates, i.e., $P_{1:t}(x_1, \ldots, x_T)=(x_1, \ldots, x_t)$. When integrating with respect to $\mu_t(dx_t)$ or $\mu_{1:t}(dx_{1:t})$, we often drop subscripts and write
\begin{align}\label{eq:abbrev_int}
    \int_{\R^d}g(x_t) \mu_t(dx_t)
    =\int_{\R^d}g(x_t)\mu(dx_t),\quad
    \int_{(\R^d)^t}g(x_{1:t}) \mu_{1:t}(dx_{1:t})
    =\int_{(\R^d)^t}g(x_{1:t})\mu(dx_{1:t})
\end{align}
if there is no confusion. Given $x\in (\R^d)^T$ we define $\mu_{x_{1:t}}\in \sp(\R^d)$, $t\in \{1,2,\ldots, T\}$ via
\begin{align}
    \mu_{x_{1:t}}(dx_{t+1})
    =\P(X_{t+1}\in dx_{t+1} \st X_{1:t}=x_{1:t})
\end{align}
for $X\sim \mu\in \sp((\R^d)^T)$. 
In other words, $(\mu_{x_{1:t}})_{x_{1:t}\in (\R^d)^t}$ are the disintegrations of $\mu_{1:t+1}$ along the first $t$ coordinates.

Recall that $\cpl(\mu, \nu)$ is the set of all couplings, $\cplbc(\mu, \nu)$ is the set of all bicausal couplings, the Wasserstein distance $\w_p$ is defined in \eqref{eq:wass} and $\a\w_p$ in Definition~\ref{df:aw}. We adopt similar shorthand notation as above for $\pi\in \sp((\R^d)^T\times (\R^d)^T)$: $\pi_{t,t}\in \sp(\R^d\times \R^d)$ (respectively, $\pi_{1:t, 1:t}\in \sp((\R^d)^t\times (\R^d)^t)$) is the pushforward measure of $\pi$ through the map $(x_1, \ldots, x_T, y_1, \ldots, y_T)\mapsto (x_t, y_t)$ (respectively, $(x_1, \ldots, x_T, y_1, \ldots, y_T)\mapsto (x_1, \ldots, x_t, y_1, \ldots, y_t)$). We use the abbreviations $\pi(dx_t, dy_t)=\pi_{t,t}(dx_t, dy_t)$ and $\pi(dx_{1:t}, dy_{1:t})=\pi_{1:t, 1:t}(dx_{1:t}, dy_{1:t})$ as in \eqref{eq:abbrev_int} if there is no confusion. Let us also remark that by definition we have
\begin{align}
    \pi_{x_{1:t}, y_{1:t}}(dx_{t+1}, dy_{t+1})
    =\P(X_{t+1}\in dx_{t+1}, Y_{t+1}\in dy_{t+1} \st X_{1:t}=x_{1:t}, Y_{1:t}=y_{1:t})
\end{align}
for $(X, Y)\sim \pi$. 

Next we introduce the convolution operation $\ast$, i.e., given two probability measures $\mu, \nu$, we define $\mu\ast \nu(A):=\int \mu(A-x)\nu(dx)$ for all measurable $A$. In particular, if $\nu$ has density $g$, then $\mu\ast \nu$ has density $\mu\ast g(x):=\int g(x-y)\mu(dy)$. The (distributional) Fourier transform of $\mu\in \sp((\R^d)^T)$, which we denote by $\widehat{\mu}: (\R^d)^T\to \C$, is given by $\widehat{\mu}(t)=\int_{(\R^d)^T} e^{-2\pi i t\cdot x}\mu(dx)$. For a function $g$, we similarly have $\widehat{g}(t):=\int_{(\R^d)^T} e^{-2\pi i t\cdot x}g(x)dx$. Recall that $\widehat{\mu\ast \nu}=\widehat{\mu}\widehat{\nu}$.

Unless otherwise stated, we reserve the notation $\xi$ for the smoothing noise distribution. For $\sigma>0$, we write $\xi_{\sigma}$ for the law of $\sigma Z$, where $Z\sim \xi$. We use the shorthand notation $\mu^{\sigma}:=\mu\ast \xi_{\sigma}$. See Section~\ref{sec:smoothaw} for details.

For any signed measure $\rho$ on $(\R^d)^T$ we define $\norm{\rho}_{\tv}:=\sup \sum_{j=1}^{\infty} \abs{\rho(E_j)}$, where the supremum is taken over all measurable partitions $(E_j)_{j\in \N}$ of $(\R^d)^T$. Recall that if $\mu, \nu\in \sp((\R^d)^T)$,
\begin{align}
    \norm{\mu-\nu}_{\tv}
    =2\sup_{G\subseteq (\R^d)^T}\left( \mu(G)-\nu(G)\right)
    =2\sup_{\pi \in \cpl(\mu, \nu)}\pi(\{x\neq y\}).
\end{align}

Given $\mu\in \R^N$ and $\Sigma\in \R^{N\times N}$, we write $\n(\mu, \Sigma)$ for the Gaussian distribution with mean vector $\mu$ and covariance matrix $\Sigma$. The $N\times N$ identity matrix is denoted as $\Id_{N}$. In particular, $\n(0, \Id_{N})$ is the standard Gaussian measure on $\R^N$. 

We write $C$ for a constant, that might change from line to line. We will detail the dependence of $C$ on non-trivial parameters throughout the paper.
\end{subsection}

\end{section}

\section{Background}\label{sec:background}
In this section, we review some auxiliary results on characterizations of relative compactness in Wasserstein distance and smooth Wasserstein distances. Proofs of all results in this section can be found in Section \ref{sec:pf_background}.

\subsection{Compactness in the adapted Wasserstein topology}\label{sec:awtopology}
Recall that $\w_p$ was introduced in \eqref{eq:wass} and $\a\w_p$ was introduced in Definition \ref{df:aw}. Given $\abs{x}\le C (\sum_{t=1}^T \abs{x_t}^p)^{1/p}$, where $C=1$ if $1\le p\le 2$ and $C=T^{1/2-1/p}$ if $p\ge 2$, we have $\w_p\le C\a\w_p$. Hence, the $\a\w_p$-relatively compact subsets of $\sp_p((\R^d)^T)$ are also $\w_p$-relatively compact. In fact, Eder showed in \cite{eder2019compactness} that the $\a\w_p$-relatively compact sets are exactly those $\w_p$-relatively compact sets whose modulus of continuity converges uniformly.

    \begin{df}[Modulus of continuity, \cite{eder2019compactness}]\label{df:modconti}
        Let $\mu\in \sp_p((\R^d)^T)$. The \emph{$p$-modulus of continuity} of $\mu$, denoted by $\omega_{\mu}^{t,p}:(0,\infty)\to [0, \infty)$, $t\in \{1,2,\ldots, T-1\}$, is given by
        \begin{align}
            \omega_{\mu}^{t,p}(\delta)
            =\sup
            \left\{\left(\E[\w_p(\mu_{X_{1:t}}, \mu_{Y_{1:t}})^p]\right)^{1/p} \,\big\vert\, X, Y\sim \mu, \left(\E[\abs{X_{1:t}-Y_{1:t}}^p]\right)^{1/p}<\delta\right\}.
        \end{align}
    \end{df}

\begingroup
\def\thethm{\ref{prop:awrelcpt}}
    \begin{prop}[Eder \cite{eder2019compactness}]
        Let $K\subseteq \sp_p((\R^d)^T)$. Then $K$ is $\a\w_p$-relatively compact if and only if $K$ is $\w_p$-relatively compact and
        \begin{align}
            \lim_{\delta\to 0}\sup_{\mu\in K}\omega^{t,p}_{\mu}(\delta)=0
            \quad \text{ for all } t\in \{1,2,\ldots, T-1\}.\label{eq:modofconti}
        \end{align}
    \end{prop}
   \addtocounter{thm}{-1}
\endgroup
    The characterization of compactness in Proposition \ref{prop:awrelcpt} is slightly different from the original result by Eder \cite{eder2019compactness}. However, it is not hard to see that they are equivalent, as we prove in Section \ref{sec:pf_background}.
    \begin{rmk}
        Trivially $K=\{\mu\}$ is $\a\w_p$-relatively compact. Hence, $\omega^{t,p}_{\mu}(\delta)\to 0$ as $\delta\to 0$.
    \end{rmk}
    \begin{rmk}
        Recall $\mu_{\ep}=\frac{1}{2}\delta_{(-\ep, -1)}+\frac{1}{2}\delta_{(\ep, 1)}$ from Example~\ref{ex:standardex}. Note that $K:=\{\mu_{\ep}: 0<\ep<1\}$ is $\w_p$-relatively compact. However, $(\mu_{\ep})_x=\delta_1$ if $x=\ep$ and $(\mu_{\ep})_x=\delta_{-1}$ if $x=-\ep$. Thus we easily compute that $\omega_{\mu_{\ep}}^{1,p}(\delta)=(\delta/\ep)\wedge 2$ and $\sup_{0<\ep<1}\omega_{\mu_{\ep}}^{1,p}(\delta)=2$. Proposition \ref{prop:awrelcpt} implies that $K$ is not $\a\w_p$-relatively compact.
    \end{rmk} 
Next we present some basic properties of the modulus of continuity $\omega$.
    \begin{prop}\label{prop:propertyofomega}
        Let $\mu\in \sp_p((\R^d)^T)$ and $t\in \{1,2,\ldots, T-1\}$.
        \begin{anumerate}
            \item\label{prop:propertyofomega_inc} The function $\omega^{t,p}_{\mu}$ is monotone increasing, i.e. $\omega^{t,p}_{\mu}(\delta_1)\le \omega^{t,p}_{\mu}(\delta_2)$ for $\delta_1<\delta_2$.
            \item\label{prop:propertyofomega_const} For $k, \delta>0$, $\omega^{t,p}_{\mu}(k\delta)\le (k\vee 1)\omega^{t,p}_{\mu}(\delta)$.
        \end{anumerate}
    \end{prop}
    If the kernels $(\R^d)^t\ni x_{1:t}\mapsto \mu_{x_{1:t}}\in \sp_p(\R^d)$ are equicontinuous,  where $\sp_p(\R^d)$ is endowed with a metric $\w_p$, then the modulus of continuity of $\mu$ converges uniformly. An important class of measures with equicontinuous kernels is given by the ones with  H\"older continuous kernels.
    \begin{df}
        Let $L>0$ and $0<\alpha\le 1$. The set $\sp_{p, \alpha, L}((\R^d)^T)$ contains all $\mu\in \sp_p((\R^d)^T)$ that satisfy
        \begin{align}
            \w_p(\mu_{x_{1:t}}, \mu_{y_{1:t}})\le L \abs{x_{1:t}-y_{1:t}}^{\alpha}
        \end{align}
        for all $x_{1:t}, y_{1:t}\in (\R^d)^t$ and $t\in \{1, 2, \ldots, T-1\}$.
    \end{df}

    \begin{prop}\label{prop:omega_holder}
        Let $L>0$ and $0<\alpha\le 1$. If $\mu\in \sp_{p, \alpha, L}((\R^d)^T)$, then $\omega^{t,p}_{\mu}(\delta)\le L \delta^{\alpha}$ for $\delta>0$ and $t\in \{1,2,\ldots, T-1\}$.
    \end{prop}
   
What is more, measures with Lipschitz kernels (i.e. $\alpha=1$) arise in many practical examples. We refer to \cite[Example 1.9.]{backhoff2022estimating} for some simple sufficient conditions.

\subsection{Smooth adapted Wasserstein distances}\label{sec:smoothaw}
As explained in Section \ref{sec:intro}, we define the smooth (adapted) Wasserstein distance as follows.
    \begin{df}[Smooth distances]
        Let $\xi \in \sp_p((\R^d)^T)$. For $\sigma>0$ and $Z\sim \xi$, denote by $\xi_{\sigma}$ the law of $\sigma Z$. We define the \emph{smooth $p$-Wasserstein distance} with the smoothing parameter $\sigma$ via
        \begin{align}
            \w^{(\sigma)}_p(\mu, \nu)
            :=\w_p(\mu\ast\xi_{\sigma}, \nu\ast \xi_{\sigma}) \text{ for } \mu, \nu\in \sp_p((\R^d)^T).
        \end{align}
        Similarly, the \emph{smooth $p$-adapted Wasserstein distance} with the smoothing parameter $\sigma$ is given by
        \begin{align}
            \a\w^{(\sigma)}_p(\mu, \nu)
            :=\a\w_p(\mu\ast \xi_{\sigma}, \nu\ast\xi_{\sigma}) \text{ for } \mu, \nu\in \sp_p((\R^d)^T).
        \end{align}
    \end{df}
    \begin{rmk}
        For $\mu\in \sp((\R^d)^T)$, we often use the shorthand notation $\mu^{\sigma}=\mu\ast \xi_{\sigma}$. In particular, $\w^{(\sigma)}_p(\mu, \nu)=\w_p(\mu^{\sigma}, \nu^{\sigma})$ and $\a\w^{(\sigma)}_p(\mu, \nu)=\a\w_p(\mu^{\sigma}, \nu^{\sigma})$.
    \end{rmk}
    In the following proposition, we present some properties of the smooth Wasserstein distance $\w^{(\sigma)}_p$. While these results are simple extensions of the classical theory, we provide proofs in Section \ref{sec:pf_background} to keep our work self-contained.
    
    \begin{prop}[Property of $\w^{(\sigma)}_p$]\label{prop:smoothW}
        Suppose that the set of real zeros of $\widehat{\xi}$ has Lebesgue measure zero. Then, for $\sigma>0$, $\w^{(\sigma)}_p$ is a metric on $\sp_p((\R^d)^T)$ that induces the same topology as $\w_p$. Furthermore, if $\sigma_1, \sigma_2>0$, then
        \begin{align}
            \big|\w^{(\sigma_1)}_p(\mu, \nu)-\w^{(\sigma_2)}_p(\mu, \nu)\big|
            \le 2\abs{\sigma_1-\sigma_2}(M_p(\xi))^{1/p}
        \end{align}
        for $\mu, \nu\in \sp_p((\R^d)^T)$. Here, recall that $M_p(\xi)=\int_{\R^{dT}} \abs{x}^p \xi(dx)$.
    \end{prop}
    The assumption that $\widehat{\xi}$ does not admit any real zeros a.s. is essential to guarantee that $\w^{(\sigma)}_p$ defines a metric. This is satisfied, for example, when $\xi$ is compactly supported (see \cite{goldfeld2024statistical}) or when $\xi$ is Gaussian. The same result is true for $\a\w^{(\sigma)}_p$.
    
    \begin{thm}[$\a\w^{(\sigma)}_p$ is a metric]\label{thm:smoothawmetric}
        Suppose that the set of real zeros of $\widehat{\xi}$ has Lebesgue measure zero. Then $\a\w_p^{(\sigma)}$ is a metric on $\sp_p((\R^d)^T)$ for $\sigma>0$.
    \end{thm}

\section{Main results}\label{sec:main results}
Now, we have all the necessary ingredients to present the main results. Proofs for this section can be found in Section \ref{sec:pf_main results}. We begin with some standing assumptions on the  noise measure $\xi$. 
    \begin{assume}\label{assume:noise}
 We have the following:
        \begin{anumerate}
            \item\label{assume:noise_f} The noise distribution $\xi\in \sp_p((\R^d)^T)$  has density $f:(\R^d)^T\to \R$ which is continuously differentiable and $\nab f\in L^1((\R^d)^T; (\R^d)^T)$. In particular, for $\sigma>0$, we write  $f_{\sigma}(\cdot):=\sigma^{-dT}f(\cdot/\sigma)$ for the density of $\xi_{\sigma}$.
            \item\label{assume:noise_FT} The set of real zeros of $\widehat{f}$ (or equivalently $\widehat{\xi}$) has Lebesgue measure zero.
        \end{anumerate}
    \end{assume}
    Assumption \ref{assume:noise_f} is a mild growth condition on $f$. Note from Theorem \ref{thm:smoothawmetric} that \ref{assume:noise_FT} guarantees that $\a\w^{(\sigma)}_p$ (as well as $\w^{(\sigma)}_p$) is indeed a metric on the space of probability measures.

    \subsection{The adapted Wasserstein distance and the total variation norm}\label{sec:aw_tv}
     Recall that the $p$-Wasserstein distance $\w_p(\mu, \nu)$ can be controlled by the $p$-total variation norm, $\norm{\abs{\cdot}^p(\mu-\nu)}_{\tv}$. For example, \cite[Theorem 6.15]{villani2009optimal} shows that
     \begin{align}
         \w_p(\mu, \nu)
         \le 2^{1-1/p}\left(\int \abs{x}^p \abs{\mu-\nu}(dx)\right)^{1/p}.
     \end{align}
     Similarly, Proposition \ref{prop:awtv} shows that $\a\w_p$ is bounded by the total variation norm $\norm{\cdot}_{\tv}$ plus remainder terms that can be made arbitrarily small. This result is based on \cite[Lemma 3.5]{eckstein2024computational} stating that the metric $\av(\mu, \nu)=\inf_{\pi\in \cplbc(\mu, \nu)}\pi(\{x\neq y\})$ is equivalent to the total variation distance $\frac{1}{2}\norm{\mu-\nu}_{\tv}=\inf_{\pi\in \cpl(\mu, \nu)}\pi(\{x\neq y\})$.
    \begin{prop}\label{prop:awtv}
    Let $\mu, \nu\in \sp_p((\R^d)^T)$. For any $R>0$ we have
        \begin{align}
            \a\w_p(\mu, \nu)^p
            \le 6^{p-1}T2^T \left(R^p \norm{\mu-\nu}_{\tv}
            +\int_{\{\abs{x}\ge R\}} \abs{x}^p (\mu+\nu)(dx)\right).
        \end{align}
    \end{prop}
    \begin{rmk}
        In Proposition~\ref{prop:awtv}, the bounding constant $6^{p-1}T2^T$ depends exponentially on $T$. This may lead to slow convergence when $T$ is large. This exponential dependence arises from \cite[Lemma 3.5]{eckstein2024computational}, which shows that
        \begin{align}
            \av(\mu, \nu)\le C\norm{\mu-\nu}_{\tv}\label{eq:av C tv}
        \end{align}
        for a constant $C$ that grows exponentially with $T$; see Section~\ref{sec:pf_aw_tv} for details. In a recent preprint, Acciaio et al. \cite{acciaio2025estimating} show that \eqref{eq:av C tv} holds with $C=2T-1$, improving the exponential bound in \cite{eckstein2024computational} to a polynomial bound. As a consequence, the constant $6^{p-1}T2^{T}$ can be replaced by one that has polynomial growth in $T$.
    \end{rmk}
    Our first main result, Theorem \ref{thm:awsigmaw1}, asserts that the smooth $p$-adapted Wasserstein distance $\a\w^{(\sigma)}_p$ can be bounded by the $1$-Wasserstein distance $\w_1$. 
    \begin{thm}\label{thm:awsigmaw1}
    Let $\mu, \nu\in \sp_p((\R^d)^T)$. For any $R>0$ and $\sigma>0$ we have
        \begin{align}
            \a\w^{(\sigma)}_p(\mu, \nu)^p
            \le 6^{p-1}T2^T \left(R^p\norm{\nab f}_{L^1}\frac{\w_1(\mu, \nu)}{\sigma}
            +\int_{\{\abs{x}\ge R\}} \abs{x}^p (\mu^{\sigma}+\nu^{\sigma})(dx)\right).
        \end{align}
    \end{thm}
    Under additional assumptions on the measures $\mu$ and $\nu$, we can actually drop the additive terms characterizing the tail behavior.

    \begin{cor}\label{cor:awsigmaw1_mom}
        For any $\sigma>0$ we have the following:
        \begin{anumerate}
            \item\label{cor:awsigmaw1_mom_q} If $\mu, \nu, \xi\in \sp_q((\R^d)^T)$ for some $q>p$, then
            \begin{align}
                \a\w^{(\sigma)}_p(\mu, \nu)^p
                \le 6^{p-1}T2^{T+1}(M_q(\mu^{\sigma}+\nu^{\sigma}))^{p/q}\left(\norm{\nab f}_{L^1}\frac{\w_1(\mu, \nu)}{\sigma}\right)^{1-p/q}.
                \label{eq:awsigmaw1_mom_q}
            \end{align}
            \item\label{cor:awsigmaw1_mom_cpt} If $\mu, \nu$ and $\xi$ have bounded supports, then
            \begin{align}
                 \a\w^{(\sigma)}_p(\mu, \nu)^p
                 \le 6^{p-1}T2^{T+2}\frac{(\diam(\supp(\mu)\cup\supp(\nu))+\sigma\diam(\supp(\xi)))^p}{\sigma}\norm{\nab f}_{L^1}\w_1(\mu, \nu).\label{eq:awsigmaw1_mom_cpt}
            \end{align}
            \item\label{cor:awsigmaw1_mom_gaussian} Suppose $\xi\sim \n(0, \Id_{dT})$. If $\mu, \nu\in \sp_q((\R^d)^T)$ for some $q>p$, then for $0<\sigma_0<\sigma$,
            \begin{align}
                \a\w^{(\sigma)}_p(\mu, \nu)^p
                \le 6^{p-1}T2^{T+1}\norm{\nab f}^{1-p/q}_{L^1}(M_q(\mu^{\sigma}+\nu^{\sigma}))^{p/q}\left(\frac{\w^{(\sigma_0)}_1(\mu, \nu)}{\sqrt{\sigma^2-\sigma^2_0}}\right)^{1-p/q}.
            \end{align}
        \end{anumerate}
    \end{cor}

    Combining Theorem \ref{thm:awsigmaw1} with the slow rate of $\w^{(\sigma)}_1$ yields the following upper bound on the convergence rate between the true and empirical measure in $\a\w^{(\sigma)}_p$. Similar bounds were announced in \cite{hou2024convergence}.

    \begin{cor}[Slow rate]\label{cor:slowrate} Let $\mu\in \sp_q((\R^d)^T)$ for some $q>p\vee (dT+2)$. Suppose $\xi\sim \n(0, \Id_{dT})$. Denote by $\boldsymbol{\mu_n}$ the empirical measure of $\mu$ with sample size $n\in \N$, i.e. $\boldsymbol{\mu_n}=\frac{1}{n}\sum_{j=1}^n \delta_{X^{(j)}}$ where $X^{(1)}, X^{(2)}, \ldots, X^{(n)}$ are i.i.d samples of $\mu$. Then
        \begin{align}
            \E[\a\w^{(\sigma)}_p(\mu, \boldsymbol{\mu_n})^p]
            \le C\left(\frac{1+\sigma^{(q+p)/2}}{\sigma^{(1+dT/2)(1-p/q)}}\right) n^{-(q-p)/(2q)}
        \end{align}
        for some constant $C>0$ that depends only on $d, T, p, q, M_q(\mu)$.
    \end{cor}

    \begin{rmk}
        If $\mu$ has a bounded support, then we have
        \begin{align}
            \E[\a\w^{(\sigma)}_p(\mu, \boldsymbol{\mu_n})^p]\le C n^{-\frac{1}{2}+\ep}
        \end{align}
        for all $\ep>0$. This recovers the slow rate $\E[\w^{(\sigma)}_p(\mu, \boldsymbol{\mu_n})^p]\le Cn^{-\frac{1}{2}}$ up to an $\ep$-loss.
    \end{rmk}

    \subsection{Smoothing and the modulus of continuity}\label{sec:bandwidth}
    Our second main result pertains to the bandwidth effect of $\mu^\sigma$. As shown in \cite[Theorem 2.4]{hou2024convergence}, the smoothing operation is stable in the sense that $\a\w_p(\mu, \mu^{\sigma})\to 0$ as $\sigma\to 0$. Here, we establish that convergence is controlled by the modulus of continuity defined in Definition \ref{df:modconti}.
        
    \begin{thm}[Bandwidth effect]\label{thm:muandmusigma} Let $\mu\in \sp_p((\R^d)^T)$. Then for any $\sigma>0$,
        \begin{align}
            \a\w_p(\mu, \mu^{\sigma})
            \le T(1\vee(M_p(\xi))^{1/p})\sum_{t=0}^{T-1}h^{t,p}_{\mu}(\sigma)
        \end{align}
        where $h^{t,p}_{\mu}$ is defined as follows
        \begin{align}
            h^{0,p}_{\mu}(\sigma)=\sigma, \quad 
            h^{t,p}_{\mu}(\sigma)=\omega_{\mu}^{t,p}\left(\sum_{s=0}^{t-1} h^{s,p}_{\mu}(\sigma)\right) \text{ for } t\in\{1,2, \ldots, T-1\}.
        \end{align}
    \end{thm}

    \begin{cor}\label{cor:uniform_moc}
         The following hold.
        \begin{anumerate}
            \item\label{cor:uniform_moc_holder} For $L>0$ and $0<\alpha\le 1$, there exists a constant $C>0$ that depends only on $T, (M_p(\xi))^{1/p}, L, \alpha$ such that
            \begin{align}
                \sup_{\mu\in \sp_{p, \alpha, L}((\R^d)^T)}\a\w_p(\mu, \mu^{\sigma})
                \le C \sum_{t=0}^{T-1}\sigma^{\alpha^{t}}.
            \end{align}
            \item\label{cor:uniform_moc_general} Let $K\subseteq \sp_p((\R^d)^T)$. If $\sup_{\mu\in K}\omega^{t,p}_{\mu}(\delta)\to 0$ as $\delta\to 0$ for all $t\in \{1,2,\ldots, T-1\}$, then
            \begin{align}
                \lim_{\sigma\searrow 0}\sup_{\mu\in K}\a\w_p(\mu, \mu^{\sigma})=0.
            \end{align}
            In particular,
            \begin{align}
                \lim_{\sigma\searrow 0}\sup_{\mu, \nu\in K}\abs{\a\w^{(\sigma)}_p(\mu, \nu)-\a\w_p(\mu, \nu)}
                =0.
            \end{align}
        \end{anumerate}
    \end{cor}

    By choosing a suitably decaying sequence $(\sigma_n)_{n\in \N}$, $((\boldsymbol{\mu_n})^{\sigma_n})_{n\in \N}$ approaches $\mu$ in $\a\w_p$. The convergence rate depends on the dimension $dT$.
    
    \begin{cor}\label{cor:empiricalvaryingsigma}
        Let $\mu\in \sp_q((\R^d)^T)\cap \sp_{p, \alpha, L}((\R^d)^T)$ for some $q>p\vee (dT+2)$, $L>0$ and $0<\alpha\le 1$. Suppose $\xi\sim \n(0, \Id_{dT})$. Using the same notation as in Corollary~\ref{cor:slowrate}, there exists a constant $C>0$ that depends only on $d, T, p, q, M_q(\mu), L, \alpha$ such that
        \begin{align}
            \E[\a\w_p(\mu, (\boldsymbol{\mu_n})^{\sigma_n})^p]
            \le C n^{-\beta}
        \end{align}
        where
        \begin{align}
            \sigma_n = n^{-\frac{1-p/q}{(dT+2)(1-p/q)+2p\alpha^{T-1}}} \text{ and } \beta = \frac{p\alpha^{T-1}(1-p/q)}{(dT+2)(1-p/q)+2p\alpha^{T-1}}.
        \end{align}
    \end{cor}

    \subsection{\texorpdfstring{Bounds on $\a\w_p$ in terms of $\w_1$}{TEXT}}\label{sec:boundingAW}
    Now, we turn our attention to bounding $\a\w_p$ in terms of $\w_1$. By introducing an auxiliary noise $\xi$ satisfying Assumption \ref{assume:noise}, a straightforward application of the triangle inequality to Theorem \ref{thm:awsigmaw1} and Theorem \ref{thm:muandmusigma} shows the following proposition.
    \begin{prop}\label{prop:boundingAWvanilla}
        Let $\mu, \nu\in \sp_p((\R^d)^T)$. Then for any $R>0$, $\sigma>0$ and $\xi$ satisfying Assumption \ref{assume:noise} we have
        \begin{align}
            \a\w_p(\mu, \nu)^p
            &\le 18^{p-1}T2^T R^p\norm{\nab f}_{L^1}\frac{\w_1(\mu, \nu)}{\sigma}
            +18^{p-1}T2^T\int_{\{\abs{x}\ge R\}} \abs{x}^p (\mu^{\sigma}+\nu^{\sigma})(dx)\\
            &+3^{p-1}T^p(1\vee M_p(\xi))\left(\sum_{t=0}^{T-1}h^{t,p}_{\mu}(\sigma)\right)^{p}
            +3^{p-1}T^p(1\vee M_p(\xi))\left(\sum_{t=0}^{T-1}h^{t,p}_{\nu}(\sigma)\right)^{p}.
        \end{align}
    \end{prop}
    We remark that $R>0$, $\sigma>0$ and $\xi$ (or equivalently $f$) can be chosen arbitrarily in the above, as long as Assumption \ref{assume:noise} is satisfied. We present several estimates as outcomes of specific choices of $R$, $\sigma$ or $\xi$.

    \begin{cor}\label{cor:boundingAW}
        Let $K\subseteq \sp_p((\R^d)^T)$ be $\a\w_p$-relatively compact. Then there exists a constant $C>0$ that depends only on $d, T, p$ such that for all $0<\sigma<R$ and $\mu, \nu\in K$,
        \begin{align}
            \a\w_p(\mu, \nu)^p
            \le C\left(R^p \frac{\w_1(\mu, \nu)}{\sigma}
            +\sup_{\gamma\in K}\left(\sum_{t=0}^{T-1}h^{t,p}_{\gamma}(\sigma)\right)^p
            +\sup_{\gamma\in K}\int_{\{\abs{x}\ge R\}}\abs{x}^p \gamma(dx)\right).
        \end{align}
    \end{cor}

    \begin{rmk}
        Recall from Proposition \ref{prop:awrelcpt} that the relative compactness of $K$ yields
        \begin{align}
            \lim_{\sigma\searrow 0}\sup_{\gamma\in K}\sum_{t=0}^{T-1}h^{t,p}_{\gamma}(\sigma)=0,\quad
            \lim_{R\to \infty}\sup_{\gamma\in K}\int_{\{\abs{x}\ge R\}}\abs{x}^p \gamma(dx)=0.
        \end{align}
    \end{rmk}

    \begin{cor}\label{cor:boundingAWcpt}
    Let $\mu, \nu\in \sp(F)$ for some bounded set $F\subseteq (\R^d)^T$.
    \begin{anumerate}
        \item\label{cor:boundingAWcpt_vanilla} There exists a constant $C>0$ that depends only on $d, T, p, \diam(F)$ such that for all $\sigma>0$,
        \begin{align}
            \a\w_p(\mu, \nu)^p
            \le C\left(\frac{(1+\sigma)^p}{\sigma}\w_1(\mu, \nu)+\left(\sum_{t=0}^{T-1}h^{t,p}_{\mu}(\sigma)\vee h^{t,p}_{\nu}(\sigma)\right)^{p}\right).
        \end{align}
        \item\label{cor:boundingAWcpt_holder} Further suppose that $\mu, \nu\in \sp_{p, \alpha, L}((\R^d)^T)$ for some $L>0$ and $0<\alpha\le 1$. Then there exists a constant $C>0$ that depends only on $d, T, p, \diam(F), L, \alpha$ such that
        \begin{align}
            \a\w_p(\mu, \nu)
            \le C\w_1(\mu, \nu)^{\alpha^{T-1}/(p\alpha^{T-1}+1)}.
        \end{align}
    \end{anumerate}
    \end{cor}

    \subsection{Comparison of topologies}\label{sec:topology}
    In this section, we complement the discussion in Section \ref{sec:awtopology} and Figure \ref{fig:implication} by giving concise mathematical statements.

    \begin{thm}[Slow decay]\label{thm:slowsigma} Let  $\mu\in \sp_p((\R^d)^T)$ and $(\mu_n)_{n\in \N}\subseteq \sp_p((\R^d)^T)$. Consider $(\sigma_n)_{n\in \N}\subseteq (0,\infty)$ such that $\sup_{n\in \N}\sigma_n<\infty$.
    \begin{anumerate}
        \item\label{thm:slowsigma_imply} If $\a\w^{(\sigma_n)}_p(\mu_n, \mu)\to 0$, then $\w_p(\mu_n, \mu)\to 0$. 
        \item\label{thm:slowsigma_reverse} If $\w_p(\mu_n, \mu)\to 0$ and $\frac{\w_1(\mu, \mu_n)}{\sigma_n}\to 0$, then $\a\w^{(\sigma_n)}_p(\mu_n, \mu)\to 0$.
    \end{anumerate}
    \end{thm}

    \begin{thm}[Fast decay]\label{thm:fastsigma} Let  $\mu\in \sp_p((\R^d)^T)$ and $(\mu_n)_{n\in \N}\subseteq \sp_p((\R^d)^T)$. Consider $(\sigma_n)_{n\in \N}\subseteq (0,\infty)$ such that $\sup_{n\in \N}\sigma_n<\infty$.
    \begin{anumerate}
        \item\label{thm:fastsigma_imply} If $\a\w_p(\mu_n, \mu)\to 0$, then $\a\w^{(\sigma_n)}_p(\mu_n, \mu)\to 0$. 
        \item\label{thm:fastsigma_reverse} If $\a\w^{(\sigma_n)}_p(\mu_n, \mu)\to 0$ and $h^{t,p}_{\mu_n}(\sigma_n)\to 0$ for all $t\in \{1,2,\ldots, T-1\}$, then $\a\w_p(\mu_n, \mu)\to 0$.
    \end{anumerate}
    \end{thm}

    \begin{thm}[Topological equivalence]\label{thm:fixedsigma}
        Let $\sigma>0$. Then $\a\w^{(\sigma)}_p$ is a metric on $\sp_p((\R^d)^T)$ that induces the same topology as $\w_p$.
    \end{thm}

In conclusion, adding independent noise $\xi$ fundamentally changes the temporal structure of stochastic processes and makes the smooth adapted Wasserstein topology different from the adapted Wasserstein topology in general. This contrasts with the Wasserstein topology, where the smooth Wasserstein distance generates the Wasserstein topology. In the case where the smoothing parameter is fixed for some $\sigma>0$, the smooth adapted Wasserstein distance completely ignores temporal evolution captured by the natural filtration.

\section{Proofs for Section \ref{sec:background}}\label{sec:pf_background}

In this section we give proofs for the results of Section \ref{sec:background}. Proofs for Section \ref{sec:awtopology} are contained in Section \ref{sec:pf_awtopolgy} and proofs for Section \ref{sec:smoothaw} can be found in Section \ref{sec:pf_smoothaw}.

\subsection{Proofs for Section \ref{sec:awtopology}}\label{sec:pf_awtopolgy}

In this section, we first state and prove Lemma \ref{lem:extmc_mc}, which is essential for proving Proposition \ref{prop:awrelcpt}. We then present proofs of Proposition \ref{prop:awrelcpt}, Proposition \ref{prop:propertyofomega} and Proposition \ref{prop:omega_holder}. We start by establishing some notation that will be used exclusively in the proofs of Lemma \ref{lem:extmc_mc} and Proposition \ref{prop:awrelcpt}.

Let $\mu\in \sp_p((\R^d)^T)$. For $t\in \{1,2,\ldots, T-1\}$ and $x_{1:t}\in (\R^d)^t$, we define $\bar{\mu}_{x_{1:t}}\in \sp_p((\R^d)^{T-t})$ via
\begin{align}
    \bar{\mu}_{x_{1:t}}(dx_{t+1:T})
    =\mu_{x_{1:T-1}}(dx_T)\cdots \mu_{x_{1:t+1}}(dx_{t+2})\mu_{x_{1:t}}(dx_{t+1}).
\end{align}
Equivalently, when $X\sim \mu$ and $A\subseteq (\R^d)^{T-t}$ is measurable,
\begin{align}
    \bar{\mu}_{x_{1:t}}(A)
    =\P(X_{t+1:T}\in A \st X_{1:t}=x_{1:t}).
\end{align}
Using this, we define $\bar{\omega}_{\mu}^{t,p}:(0,\infty)\to [0,\infty)$ via
\begin{align}
    \bar{\omega}_{\mu}^{t,p}(\delta)
    =\sup
    \left\{\left(\E[\w_p(\bar{\mu}_{X_{1:t}}, \bar{\mu}_{Y_{1:t}})^p]\right)^{1/p} \,\big\vert\, X, Y\sim \mu, \left(\E[\abs{X_{1:t}-Y_{1:t}}^p]\right)^{1/p}<\delta\right\}.
\end{align}

\begin{lem}\label{lem:extmc_mc}
    Let $\mu\in \sp_p((\R^d)^T)$. For $\delta>0$ and $t\in \{1,2,\ldots, T-1\}$, we have
    \begin{align}
        \bar{\omega}_{\mu}^{t,p}(\delta)
        \le \sum_{s=t}^{T-1} g_{t,\mu}^{s,p}(\delta)
    \end{align}
    where
    \begin{align}
        g_{t,\mu}^{s,p}(\delta)
        :=\omega_{\mu}^{s,p}\left(\delta+\sum_{\ell=t}^{s-1} g_{t,\mu}^{\ell, p}(\delta)\right) \text{ for } s\in \{t,t+1 \ldots, T-1\}.
    \end{align}
    Here, the sum $\sum_{\ell=t}^{s-1}$ is understood as $0$ when $s=t$.
\end{lem}
\begin{proof}
    Let us fix $t\in \{1,2,\ldots, T-1\}$ and $\pi\in \cpl(\mu_{1:t}, \mu_{1:t})$ such that
    \begin{align}
        \left(\int \abs{x_{1:t}-y_{1:t}}^p \pi(dx_{1:t}, dy_{1:t})\right)^{1/p}<\delta.\label{eq:piext}
    \end{align}
    For $s\in \{t, t+1, \ldots, T-1\}$ and $x_{1:s}, y_{1:s}\in (\R^d)^s$, define $\pi^{\star}_{x_{1:s}, y_{1:s}}\in \cpl(\mu_{x_{1:s}}, \mu_{y_{1:s}})$ as a $\w_p$-optimal coupling between $\mu_{x_{1:s}}$ and $\mu_{y_{1:s}}$. Note that $\gamma\in \cpl(\mu, \mu)$ which is defined via
    \begin{align}
        &\gamma(dx_{1:T}, dy_{1:T})\\
        &:=\pi^{\star}_{x_{1:T-1}, y_{1:T-1}}(dx_T, dy_T)\cdots \pi^{\star}_{x_{1:t+1}, y_{1:t+1}}(dx_{t+2}, dy_{t+2})\pi^{\star}_{x_{1:t}, y_{1:t}}(dx_{t+1}, dy_{t+1})\pi(dx_{1:t}, dy_{1:t}).
    \end{align}
    Let $(X, Y)\sim \gamma$. From the construction of $\gamma$, we bound
    \begin{align}
        \left(\E[\abs{X_{s+1}-Y_{s+1}}^p]\right)^{1/p}
        &=\left(\int \abs{x_{s+1}-y_{s+1}}^p \gamma(dx_{1:s+1}, dy_{1:s+1})\right)^{1/p}\\
        &=\left(\int \w_p(\mu_{x_{1:s}}, \mu_{y_{1:s}})^p \gamma(dx_{1:s}, dy_{1:s})\right)^{1/p}\\
        &\le \omega_{\mu}^{s, p}(\left(\E[\abs{X_{1:s}-Y_{1:s}}^p]\right)^{1/p})\\
        &\le\omega_{\mu}^{s,p}\left(\delta+\sum_{\ell=t}^{s-1}\left(\E[\abs{X_{\ell+1}-Y_{\ell+1}}^p]\right)^{1/p}\right).\label{eq:recursiveg}
    \end{align}
    Here the first inequality comes from the definition of $\omega_{\mu}^{s,p}$ and the last inequality is due to the inequality \eqref{eq:piext},
    \begin{align}
        (\E[\abs{X_{1:t}-Y_{1:t}}^p])^{1/p}
        =\left(\int \abs{x_{1:t}-y_{1:t}}^p \pi(dx_{1:t}, dy_{1:t})\right)^{1/p}
        <\delta.
    \end{align}
    Using the recursive formula \eqref{eq:recursiveg}, we obtain that $\left(\E[\abs{X_{s+1}-Y_{s+1}}^p]\right)^{1/p}\le g_{t,\mu}^{s,p}(\delta)$. Observe that
    \begin{align}
        \left(\int \w_p(\bar{\mu}_{x_{1:t}}, \bar{\mu}_{y_{1:t}})^p \pi(dx_{1:t}, dy_{1:t})\right)^{1/p}
        &\le \left(\E[\abs{X_{t+1:T}-Y_{t+1:T}}^p]\right)^{1/p}\\
        &\le \sum_{s=t}^{T-1}\left(\E[\abs{X_{s+1}-Y_{s+1}}^p]\right)^{1/p}
        \le \sum_{s=t}^{T-1}g_{t,\mu}^{s,p}(\delta).\label{eq:extmcbound}
    \end{align}
    Since the above estimate \eqref{eq:extmcbound} holds for all $\pi\in\cpl(\mu_{1:t}, \mu_{1:t})$ satisfying \eqref{eq:piext}, taking the supremum over all such $\pi$ in the left hand side of \eqref{eq:extmcbound} gives the desired estimate.
\end{proof}

    \begin{proof}[Proof of Proposition \ref{prop:awrelcpt}]
        It follows from \cite[Theorem $3.2$]{eder2019compactness} that $K\subseteq \sp_p((\R^d)^T)$ is $\a\w_p$-relatively compact if and only if $K$ is $\w_p$-relatively compact and
        \begin{align}
            \lim_{\delta\to 0} \sup_{\mu\in K} \bar{\omega}_{\mu}^{t,p}(\delta)=0 \text{ for all } t\in \{1,2,\ldots, T-1\}\label{eq:extmodofconti}
        \end{align}
        It is straightforward from the inequality $\w_p(\mu_{x_{1:t}}, \mu_{y_{1:t}}) \le \w_p(\bar{\mu}_{x_{1:t}}, \bar{\mu}_{y_{1:t}})$ for all $x_{1:t}, y_{1:t}\in (\R^d)^t$, that condition \eqref{eq:extmodofconti} implies condition \eqref{eq:modofconti}. The fact that condition \eqref{eq:modofconti} implies condition \eqref{eq:extmodofconti} follows from Lemma \ref{lem:extmc_mc}.
    \end{proof}

    \begin{proof}[Proof of Proposition \ref{prop:propertyofomega}]
    \ref{prop:propertyofomega_inc} The proof is obvious from the definition of $\omega^{t,p}_{\mu}$.
    
    \ref{prop:propertyofomega_const}
        If $k\le 1$, the proof is immediate from the fact that a map $\omega_{\mu}^{t,p}(\cdot)$ is nondecreasing. Now let us assume $k>1$. The proof for the case $k>1$ is a simple re-scaling of \cite[Lemma $2.4$]{eder2019compactness}. For the sake of simplicity, let us write $\x=(\R^d)^t$, $\y=\sp(\R^d)$ and define a map $\L_t:\sp((\R^d)^T)\to \sp(\x\times \y)$ via $\L_t(\nu)=\law(Y_{1:t}, \nu_{Y_{1:t}})$ for $Y\sim \nu$. As a byproduct of \cite[Lemma $2.4$]{eder2019compactness}, we can relax the marginal constraints as follows:
        \begin{align}
            \omega_{\mu}^{t,p}(\delta)
            =\sup_{\gamma\in \per(\L_t(\mu), \delta)} \rho^{\y}(\gamma)
        \end{align}
        where
        \begin{align}
            &\rho^{\x}(\gamma)=\left(\int \abs{x_{1:t}-y_{1:t}}^p \gamma(dx_{1:t}, d\kappa, dy_{1:t}, d\rho)\right)^{1/p},\\
            &\rho^{\y}(\gamma)=\left(\int \w_p(\kappa, \rho)^p \gamma(dx_{1:t}, d\kappa, dy_{1:t}, d\rho)\right)^{1/p},\\
            &\per(\L_t(\mu), \delta)
            =\left\{\gamma\in \sp^{\le}(\x\times\y\times\x\times\y) :\, \rho^{\x}(\gamma)<\delta, P^{(1)}_{\#}\gamma\le \L_t(\mu), P^{(2)}_{\#}\gamma\le \L_t(\mu)\right\}.
        \end{align}
        Here, $\sp^{\le}(\x\times\y\times\x\times\y)$ is the collection of positive measures whose total mass is no greater than $1$ and the maps $P^{(j)}:\x\times\y\times\x\times\y\to \x\times \y$ are defined via $P^{(j)}(x_1, y_1, x_2, y_2):=(x_j, y_j)$ for $j\in \{1,2\}$. Hence if $\gamma\in \per(\L_t(\mu), k\delta)$, then $k^{-p}\gamma\in \per(\L_t(\mu), \delta)$ and $\rho^{\y}(\gamma)=k\rho^{\y}(k^{-p}\gamma)$. This implies that $\omega_{\mu}^{t,p}(k\delta)\le k\omega_{\mu}^{t,p}(\delta)$.
    \end{proof}

    \begin{proof}[Proof of Proposition \ref{prop:omega_holder}]
        Let $\mu\in \sp_{p, \alpha, L}((\R^d)^T)$. Suppose $X, Y\sim \mu$ and $\E[\abs{X_{1:t}-Y_{1:t}}^p]<\delta^{p}$. Applying Jensen's inequality we obtain
        \begin{align}
            \left(\E[\w_p(\mu_{X_{1:t}}, \mu_{Y_{1:t}})^p]\right)^{1/p}
            \le L\left(\E[\abs{X_{1:t}-Y_{1:t}}^{p\alpha}]\right)^{1/p}
            \le L \left(\E[\abs{X_{1:t}-Y_{1:t}}^{p}]\right)^{\alpha/p}
            \le L \delta^{\alpha}.
        \end{align}
        Hence, $\omega^{t,p}_{\mu}(\delta)\le L \delta^{\alpha}$.
    \end{proof}

\subsection{Proofs for Section \ref{sec:smoothaw}}\label{sec:pf_smoothaw}
    \begin{proof}[Proof of Proposition \ref{prop:smoothW} and Theorem \ref{thm:smoothawmetric}]
         The proof is a straightforward generalization of \cite{nietert2021smooth, goldfeld2024limit, goldfeld2024statistical}: the case where the noise $\xi$ is Gaussian is treated in \cite{nietert2021smooth, goldfeld2024limit} and \cite{goldfeld2024statistical} shows the same result when $\xi$ is compactly supported.
         
         The fact that $\w^{(\sigma)}_p$ and $\a\w^{(\sigma)}_p$ satisfy symmetry and the triangle inequality is trivial; for example, see \cite[Lemma 2]{nietert2021smooth}. If $\w^{(\sigma)}_p(\mu, \nu)=0$ or $\a\w^{(\sigma)}_p(\mu, \nu)=0$, then $\mu\ast \xi_{\sigma}=\nu\ast \xi_{\sigma}$. By taking Fourier transforms, we deduce that
         \begin{align}
             \widehat{\mu}(t)\widehat{\xi}(\sigma t)
             =\widehat{\mu\ast \xi_{\sigma}}(t)
             =\widehat{\nu\ast \xi_{\sigma}}(t)=\widehat{\nu}(t)\widehat{\xi}(\sigma t) \text{ for all } t\in (\R^d)^T.
         \end{align}
         Since the set of real zeros of $\widehat{\xi}$ has Lebesgue measure zero, $\widehat{\mu}=\widehat{\nu}$ a.s. In particular, $\widehat{\mu}=\widehat{\nu}$ as tempered distributions. The Fourier inversion formula shows that $\mu=\nu$ as tempered distributions and thus $\mu=\nu$ as measures. This shows that $\w^{(\sigma)}_p$ and $\a\w^{(\sigma)}_p$ are metrics on $\sp_p((\R^d)^T)$. For a proof that $\w^{(\sigma)}_p$ metrizes the $\w_p$-topology we refer to \cite[Lemma 15]{goldfeld2024statistical}. Lastly, for $\sigma_1, \sigma_2>0$, find $(X, Z)\sim \mu\otimes \xi$ and $(Y, W)\sim \nu\otimes \xi$ such that $(X+\sigma_1 Z, Y+\sigma_1 W)$ is an optimal coupling of $\w^{(\sigma_1)}_p(\mu, \nu)$. Then by the triangle inequality,
         \begin{align}
             \w^{(\sigma_2)}_p(\mu, \nu)
             &\le \left(\E[\abs{(X+\sigma_2 Z)-(Y+\sigma_2 W)}^p]\right)^{1/p}\\
             &\le \left(\E[\abs{(X+\sigma_1 Z)-(Y+\sigma_1 W)}^p]\right)^{1/p}
             +\abs{\sigma_1-\sigma_2}\left(\E[\abs{Z-W}^p]\right)^{1/p}\\
             &\le \w^{(\sigma_1)}_p(\mu, \nu)+2\abs{\sigma_1-\sigma_2}(M_p(\xi))^{1/p}.
         \end{align}
         This proves desired estimates.
    \end{proof}

\section{Proofs for Section \ref{sec:main results}}\label{sec:pf_main results}
Proofs for Sections~\ref{sec:aw_tv}, \ref{sec:bandwidth}, \ref{sec:boundingAW} and \ref{sec:topology} are contained in Sections \ref{sec:pf_aw_tv}, \ref{sec:pf_bandwidth}, \ref{sec:pf_boundingAW} and \ref{sec:pf_topology}.

\subsection{Proofs for Section \ref{sec:aw_tv}}\label{sec:pf_aw_tv}
This section is organized as follows. In Lemma \ref{lem:cptapprox}, we show that compactly supported measures are $\a\w_p$-dense via an explicit estimate. Combined with Lemma \ref{lem:tvav}, this proves Proposition \ref{prop:awtv}. The proof of Theorem \ref{thm:awsigmaw1} is based on Lemma \ref{lem:tvsigmaw1}. Following that, we prove Corollary \ref{cor:awsigmaw1_mom} and Corollary \ref{cor:slowrate}.

\begin{lem}[Approximation by compactly supported measures]\label{lem:cptapprox}
    Let $R>0$. There exists a Lipschitz function $\Phi^{(R)}=(\Phi^{(R)}_1, \ldots, \Phi^{(R)}_T):(\R^d)^T\to (\R^d)^T$ such that
    \begin{align}
        \a\w_p(\mu, \Phi^{(R)}_{\#}\mu)
        \le 2T^{2/p}\left(\int_{\{\abs{x}\ge R\}} \abs{x}^p \mu(dx)\right)^{1/p}
    \end{align}
    for any $\mu\in \sp_p((\R^d)^T)$. Moreover, $\abs{\Phi^{(R)}_t}\le R$ for every $t\in \{1,2,\ldots, T\}$.
\end{lem}
\begin{proof}
    Throughout the proof, we assume that $R>0$ and $\mu\in \sp_p((\R^d)^T)$ are fixed. We define $\Phi:(\R^d)^T\to (\R^d)^T$ via $\Phi(x_{1:T}):=(\varphi(x_1), \ldots, \varphi(x_T))$, where $\varphi:\R^d\to \R^d$ is given by
    \begin{align}
        \varphi(x)
        =\begin{cases}
            x & \text{ if } \abs{x}< R\\
            \frac{Rx}{\abs{x}} & \text{ if } \abs{x}\ge R.
        \end{cases}
    \end{align}
    We want to construct a bicausal coupling between $\mu$ and $\Phi_{\#}\mu$. Set $\pi_{1,1}:=(\id, \varphi)_{\#}\mu_1\in \cpl(\nu_1, (\Phi_{\#}\nu)_1)$. For $t\in \{1,2,\ldots, T-1\}$ and $x_{1:t}, y_{1:t}\in (\R^d)^t$, let us define $\pi_{x_{1:t}, y_{1:t}}\in \sp(\R^d\times \R^d)$ via
    \begin{align}
        \pi_{x_{1:t}, y_{1:t}}
        :=\begin{cases}
            (\id, \varphi)_{\#}\mu_{x_{1:t}} & \text{ if } x_{1:t}=y_{1:t} \text{ and } \abs{x_{\ell}}<R \text{ for all } \ell\in \{1,2,\ldots, t\},\\
            \mu_{x_{1:t}}\otimes (\Phi_{\#}\mu)_{y_{1:t}} & \text{ otherwise.} 
        \end{cases}
    \end{align}
    Since $(\Phi_{\#}\mu)_{y_{1:t}}=\varphi_{\#}\mu_{x_{1:t}}$ on $\{x_{1:t}=y_{1:t} \text{ and } \abs{x_{\ell}}<R \text{ for all } \ell\in \{1,2,\ldots, t\}\}$, we have that $\pi_{x_{1:t}, y_{1:t}}\in \cpl(\mu_{x_{1:t}}, (\Phi_{\#}\mu)_{y_{1:t}})$. Define the measure $\pi\in \sp((\R^d)^T\times (\R^d)^T)$ via 
    \begin{align}
        \pi(dx, dy)
        :=\pi_{x_{1:T-1}, y_{1:T-1}}(dx_T, dy_T)\cdots \pi_{x_{1}, y_{1}}(dx_2, dy_2)\pi_{1,1}(dx_1, dy_1)
    \end{align}
    and conclude that $\pi\in \cplbc(\mu, \Phi_{\#}\mu)$. Thus
    \begin{align}
        \a\w_p(\mu, \Phi_{\#}\mu)^p
        \le \sum_{t=1}^T\int \abs{x_t-y_t}^p \pi(dx_{1:t}, dy_{1:t}).
    \end{align}
    For $t\in \{1, 2, \ldots, T\}$, consider the partition $(E^t_{\ell})_{\ell=0}^t$ of $(\R^d)^t\times (\R^d)^t$ where
    \begin{align}
        &E^t_0:=\{(x_{1:t}, y_{1:t}):\, \abs{x_s}<R \text{ for all } s\in \{1,2, \ldots, t\}\},\\
        &E^t_{\ell}
        :=\{(x_{1:t}, y_{1:t}):\, \abs{x_{\ell}}\ge R \text{ and } \abs{x_s}< R \text{ for all } s\in \{1,2, \ldots, \ell\}\} \text{ when } \ell\in\{1,2,\ldots, t\}.
    \end{align}
    Note that
    \begin{align}
        \int_{E^t_0} \abs{x_t-y_t}^p \pi(dx_{1:t}, dy_{1:t})
        =\int_{\{ \abs{x_t}<R \}} \abs{x_t-\varphi(x_t)}^p \mu(x_{1:t})=0.
    \end{align}
    For $\ell\in \{1,2, \ldots, t\}$, we use Jensen's inequality to compute
    \begin{align}
        \int_{E^t_{\ell}} \abs{x_t-y_t}^p \pi(dx_{1:t}, dy_{1:t})
        &\le 2^{p-1}\int_{E^t_{\ell}} \abs{x_t}^p \mu(dx_{1:t})
        +2^{p-1}\int_{E^t_{\ell}} \abs{y_t}^p \Phi_{\#}\mu(dy_{1:t})\\
        &\le 2^{p}\int_{\{\abs{x_{\ell}}\ge R\}} \abs{x_t}^p \mu(dx_{1:t}).
    \end{align}
    Combining these results,
    \begin{align}
        \a\w_p(\mu, \Phi_{\#}\mu)^p
        &\le \sum_{t=1}^T \sum_{\ell=1}^{t}
        \int_{E^t_{\ell}} \abs{x_t-y_t}^p \pi(dx_{1:t}, dy_{1:t})\\
        &\le \sum_{t=1}^T \sum_{\ell=1}^{t}
        2^{p}\int_{\{\abs{x_{\ell}}\ge R\}} \abs{x_t}^p \mu(dx_{1:t})
        \le 
        2^pT^2\int_{\{\abs{x}\ge R\} } \abs{x}^p \mu(dx).
    \end{align}
\end{proof}

    \begin{lem}\label{lem:tvav}
    For $\mu, \nu\in \sp((\R^d)^T)$ we have
    \begin{align}
        \frac{1}{2}\norm{\mu-\nu}_{\tv}
        \le \av(\mu, \nu)
        :=\inf_{\pi \in \cplbc(\mu, \nu)} \pi(\{x\neq y\})
        \le \frac{2^T-1}{2}\norm{\mu-\nu}_{\tv}.
    \end{align}
\end{lem}
\begin{proof}
    See \cite[Lemma 3.5]{eckstein2024computational}.
\end{proof}

    \begin{proof}[Proof of Proposition \ref{prop:awtv}]
        Fix $R>0$ and define $\Phi^{(R)}$ as in Lemma \ref{lem:cptapprox}. Observe that
        \begin{align}
            \a\w_p(\Phi^{(R)}_{\#}\mu, \Phi^{(R)}_{\#}\nu)^p
            &\le(2R)^p\inf_{\pi \in \cplbc(\Phi^{(R)}_{\#}\mu, \Phi^{(R)}_{\#}\nu)}\int \sum_{t=1}^T \ind{\{x_t\neq y_t\}} \pi(dx_{1:T}, dy_{1:T})\\
            &\le (2R)^p T \inf_{\pi \in \cplbc(\Phi^{(R)}_{\#}\mu, \Phi^{(R)}_{\#}\nu)}\int \ind{\{x_{1:T}\neq y_{1:T}\}} \pi(dx_{1:T}, dy_{1:T})\\
            &= (2R)^p T \av(\Phi^{(R)}_{\#}\mu, \Phi^{(R)}_{\#}\nu)\\
            &\le (2R)^p T((2^T-1)/2) \big\|\Phi^{(R)}_{\#}\mu-\Phi^{(R)}_{\#}\nu\big\|_{\tv}\\
            &\le (2R)^p T((2^T-1)/2) \norm{\mu-\nu}_{\tv}.\label{eq:cptawtv}
        \end{align}
        The first inequality stems from the inclusion $\{x_t\neq y_t\}\subseteq \{x_{1:T}\neq y_{1:T}\}$. Lemma \ref{lem:tvav} gives the second inequality. The last inequality is straightforward from the definition of $\norm{\cdot}_{\tv}$. Indeed,
        \begin{align}
            \frac{1}{2}\big\|\Phi^{(R)}_{\#}\mu-\Phi^{(R)}_{\#}\nu\|_{\tv}
            &=\sup_{A\subseteq (\R^d)^T} \mu((\Phi^{(R)})^{-1}(A))-\nu((\Phi^{(R)})^{-1}(A))\\
            &\le \sup_{B\subseteq (\R^d)^T} \mu(B)-\nu(B)
            =\frac{1}{2}\norm{\mu-\nu}_{\tv}.
        \end{align}
        Using the triangle inequality together with the estimate \eqref{eq:cptawtv} and Lemma \ref{lem:cptapprox},
        \begin{align}
            \a\w_p(\mu, \nu)^p
            &\le 3^{p-1}\left(\a\w_p(\mu, \Phi^{(R)}_{\#}\mu)^p+\a\w_p(\nu, \Phi^{(R)}_{\#}\nu)^p+\a\w_p(\Phi^{(R)}_{\#}\mu, \Phi^{(R)}_{\#}\nu)^p\right)\\
            &\le 3^{p-1}2^p T^2 \int_{\{\abs{x}\ge R\}}\abs{x}^p (\mu+\nu)(dx)
            +3^{p-1}(2R)^p T((2^T-1)/2) \norm{\mu-\nu}_{\tv}\\
            &\le (3^{p-1}2^p T^2)\vee (3^{p-1}2^p T((2^T-1)/2)) \left(\int_{\{\abs{x}\ge R\}}\abs{x}^p (\mu+\nu)(dx)
            +R^p\norm{\mu-\nu}_{\tv}\right)\\
            &\le 6^{p-1}T2^T\left(\int_{\{\abs{x}\ge R\}}\abs{x}^p (\mu+\nu)(dx)+R^p\norm{\mu-\nu}_{\tv}\right).
        \end{align}
        This proves the desired estimate.
    \end{proof}

    \begin{lem}\label{lem:tvsigmaw1}
        For $\mu, \nu\in \sp_p((\R^d)^T)$ we have
        \begin{align}
            \norm{\mu^{\sigma}-\nu^{\sigma}}_{\tv}
            \le \norm{\nab f}_{L^1}\frac{\w_1(\mu, \nu)}{\sigma}.
        \end{align}
    \end{lem}
    \begin{proof}
        We first claim that
        \begin{equation}\label{eq:claim_tv}
            \begin{aligned}
            \norm{\mu^{\sigma}-\nu^{\sigma}}_{\tv}
            &=\sup\left\{\int \varphi d\mu^{\sigma}-\int \varphi d\nu^{\sigma} :\, \abs{\varphi}\le 1\right\}\\
            &=\sup\left\{\int \varphi d\mu^{\sigma}-\int \varphi d\nu^{\sigma} :\,  \abs{\varphi}\le 1, \diam(\supp(\varphi))<\infty\right\}.
        \end{aligned}
        \end{equation}
        The first equality is classical, see for example \cite[Chapter $3$]{folland1999real}. To prove the second equality, let us fix $\ep>0$ and choose $R>0$ such that $(\mu^{\sigma}+\nu^{\sigma})(\{x\in (\R^d)^T:\, \abs{x}>R\})\le \ep$. Take any function $\varphi:(\R^d)^T\to \R$ such that $\abs{\varphi}\le 1$, set $\varphi_{R}(x):=\varphi(x)\ind{\{\abs{x}\le R\}}$ and observe that
        \begin{align}
            \int \varphi d\mu^{\sigma}-\int \varphi d\nu^{\sigma}
            &=\left(\int \varphi_R d\mu^{\sigma}-\int \varphi_R d\nu^{\sigma}\right)
            +\left(\int [\varphi-\varphi_R] d\mu^{\sigma}-\int [\varphi-\varphi_R] d\nu^{\sigma}\right)\\
            &\le \left(\int \varphi_R d\mu^{\sigma}-\int \varphi_R d\nu^{\sigma}\right)
            +\ep.
        \end{align}
        Since $\varphi$ is arbitrary, it shows the claim.

        Now, let $\varphi:(\R^d)^T\to \R$ such that $\abs{\varphi}\le 1$ and $\diam(\supp(\varphi))<\infty$. Note that
        \begin{align}
            \int \varphi d\mu^{\sigma}-\int \varphi d\nu^{\sigma}
            =\int \left(\int \varphi(x)f_{\sigma}(x-y)dx\right)\mu(dy)-\int \left(\int \varphi(x)f_{\sigma}(x-y)dx\right)\nu(dy).
        \end{align}
        Since $\varphi$ has bounded support and $f$ is continuously differentiable, the dominated convergence theorem implies that
        \begin{align}
            \nab_{y}\left(\int \varphi(x)f_{\sigma}(x-y)dx\right)
            =-\int \varphi(x) (\nab f_{\sigma})(x-y)dx.
        \end{align}
        In particular
        \begin{align}
            \abs{\nab_{y}\left(\int \varphi(x)f_{\sigma}(x-y)dx\right)}
            \le \int \abs{\nab f_{\sigma}(x)}dx
            =\frac{1}{\sigma}\int \abs{\nab f(x)}dx.
        \end{align}
        Thus, the function $y\mapsto \int \varphi(x)f_{\sigma}(x-y)dx$ is Lipschitz continuous with the Lispchitz constant $\norm{\nab f}_{L^1}/\sigma$. The desired estimate follows from the Kantorovich--Rubinstein duality for $\w_1(\mu, \nu)$ and \eqref{eq:claim_tv}.
    \end{proof}

    \begin{proof}[Proof of Theorem \ref{thm:awsigmaw1}]
    Combining Proposition \ref{prop:awtv} and Lemma \ref{lem:tvsigmaw1},
        \begin{align}
            \a\w_p(\mu^{\sigma}, \nu^{\sigma})^p
            &\le 6^{p-1}T2^T \left(R^p \norm{\mu^{\sigma}-\nu^{\sigma}}_{\tv}
            +\int_{\{\abs{x}\ge R\}} \abs{x}^p (\mu^{\sigma}+\nu^{\sigma})(dx)\right)\\
            &\le 6^{p-1}T2^T \left(R^p \norm{\nab f}_{L^1}\frac{\w_1(\mu, \nu)}{\sigma}
            +\int_{\{\abs{x}\ge R\}} \abs{x}^p (\mu^{\sigma}+\nu^{\sigma})(dx)\right).
        \end{align}
    \end{proof}

    \begin{proof}[Proof of Corollary \ref{cor:awsigmaw1_mom}]
        \ref{cor:awsigmaw1_mom_q} For the sake of simplicity, let us write $A_{R}:=\{x\in (\R^d)^T :\, \abs{x}\ge R\}$. Using the Markov inequality, we bound
        \begin{align}
            \int_{\{\abs{x}\ge R\}}\abs{x}^p (\mu^{\sigma}+\nu^{\sigma})(dx)
            &=\int \abs{x}^p \ind{A_{R}}(x)(\mu^{\sigma}+\nu^{\sigma})(dx)\\
            &\le \int \abs{x}^p \frac{\abs{x}^{q-p}}{R^{q-p}}(\mu^{\sigma}+\nu^{\sigma})(dx)
            =\frac{M_q(\mu^{\sigma}+\nu^{\sigma})}{R^{q-p}}.
        \end{align}
        Plugging this into Theorem \ref{thm:awsigmaw1},
        \begin{align}
            \a\w^{(\sigma)}_p(\mu, \nu)^p
            \le 6^{p-1}T2^T \left(R^p\norm{\nab f}_{L^1}\frac{\w_1(\mu, \nu)}{\sigma}
            +\frac{M_q(\mu^{\sigma}+\nu^{\sigma})}{R^{q-p}}\right).
        \end{align}
        The desired estimate \eqref{eq:awsigmaw1_mom_q} follows from optimizing the right-hand side by choosing
        \begin{align}
            R=\left(\frac{\sigma M_q(\mu^{\sigma}+\nu^{\sigma})}{\norm{\nab f}_{L^1}\w_1(\mu, \nu)}\right)^{1/q}.
        \end{align}

        \ref{cor:awsigmaw1_mom_cpt} Without loss of generality, we may assume that $0\in (\supp(\mu)\cup\supp(\nu))\cap\supp(\xi)$. The proof for the general case follows from translations. Applying the bound
        \begin{align}
            (M_q(\mu^{\sigma}+\nu^{\sigma}))^{p/q}
            &\le (M_q(\mu^{\sigma}))^{p/q}
            +(M_q(\nu^{\sigma}))^{p/q}\\
            &\le ((M_q(\mu))^{1/q}+\sigma (M_q(\xi))^{1/q})^p
            +((M_q(\nu))^{1/q}+\sigma (M_q(\xi))^{1/q})^p\\
            &\le 2(\diam(\supp(\mu)\cup\supp(\nu))+\sigma\diam(\supp(\xi)))^p
        \end{align}
        and sending the limit $q\to \infty$ in \eqref{eq:awsigmaw1_mom_q}, we obtain
        \begin{align}
            \a\w^{(\sigma)}_p(\mu, \nu)^p
            &\le 6^{p-1}T2^{T+1}\liminf_{q\to \infty}(M_q(\mu^{\sigma}+\nu^{\sigma}))^{p/q}\left(\norm{\nab f}_{L^1}\frac{\w_1(\mu, \nu)}{\sigma}\right)^{1-p/q}\\
            &\le 6^{p-1}T2^{T+2}(\diam(\supp(\mu)\cup\supp(\nu))+\sigma\diam(\supp(\xi)))^p\left(\norm{\nab f}_{L^1}\frac{\w_1(\mu, \nu)}{\sigma}\right).
        \end{align}

        \ref{cor:awsigmaw1_mom_gaussian} If $\xi$ is a Gaussian measure we have $\xi_{\sigma}=(\xi_{\sigma_0})\ast (\xi_{\sqrt{\sigma^2-\sigma^2_0}})$. Since $q>p$, we apply Corollary \ref{cor:awsigmaw1_mom}\ref{cor:awsigmaw1_mom_q} to deduce that
        \begin{align}
            \a\w^{(\sigma)}_p(\mu, \nu)^p
            &=\a\w^{(\sqrt{\sigma^2-\sigma^2_0})}_p(\mu^{\sigma_0}, \nu^{\sigma_0})^p\\
            &\le 6^{p-1}T2^{T+1}(M_q(\mu^{\sigma}+\nu^{\sigma}))^{p/q}\left(\norm{\nab f}_{L^1}\frac{\w^{(\sigma_0)}_1(\mu, \nu)}{\sqrt{\sigma^2-\sigma^2_0}}\right)^{1-p/q}.
        \end{align}
    \end{proof}

    \begin{proof}[Proof of Corollary \ref{cor:slowrate}]
        Choosing $\sigma_0=\sigma/\sqrt{2}$ in Corollary \ref{cor:awsigmaw1_mom}\ref{cor:awsigmaw1_mom_gaussian},
        \begin{align}
            \a\w^{(\sigma)}_p(\mu, \boldsymbol{\mu_n})^p
            \le 6^{p-1}T2^{T+1}\norm{\nab f}^{1-p/q}_{L^1}(M_q(\mu^{\sigma}+\boldsymbol{\mu_n}^{\sigma}))^{1/r}\left(\frac{\w^{(\sigma/\sqrt{2})}_1(\mu, \boldsymbol{\mu_n})}{\sigma/\sqrt{2}}\right)^{1/r'}
        \end{align}
        where $r=q/p>1$ and $r'$ is the H\"older conjugate of $r$. Applying H\"older's inequality with $1/r+1/r'=1$, 
        \begin{align}
            \E[\a\w^{(\sigma)}_p(\mu, \boldsymbol{\mu_n})^p]
            &\le C(M_q(\mu^{\sigma})+\E[M_q(\boldsymbol{\mu_n}^{\sigma})])^{1/r}\left(\frac{\E[\w^{(\sigma/\sqrt{2})}_1(\mu, \boldsymbol{\mu_n})]}{\sigma/\sqrt{2}}\right)^{1/r'}\\
            &\le C(M_q(\mu)+\sigma^q M_q(\xi))^{p/q}\left(\frac{\E[\w^{(\sigma/\sqrt{2})}_1(\mu, \boldsymbol{\mu_n})]}{\sigma}\right)^{1/r'}\\
            &\le C(1+\sigma^p)\left(\frac{\E[\w^{(\sigma/\sqrt{2})}_1(\mu, \boldsymbol{\mu_n})]}{\sigma}\right)^{1/r'}
        \end{align}
        for some constant $C>0$ that depends only on $d, T, p, q, M_q(\mu)$. A close inspection of the proof of \cite[Theorem 3.1]{hou2024convergence} shows that if $q>dT+2$, then for $X\sim \mu$,
        \begin{align}
            \E[\w_1^{(\sigma/\sqrt{2})}(\mu, \boldsymbol{\mu_n})]
            &\le \int_{(\R^d)^T} \abs{x}\E\left[\abs{f_{\sigma/\sqrt{2}}\ast \mu(x)-f_{\sigma/\sqrt{2}}\ast \boldsymbol{\mu_n}(x)}\right]dx\\
            &\le \left(\int_{(\R^d)^T}\frac{\abs{x}^2}{1+\abs{x}^{q}} dx\right)^{1/2}
            \left(\int_{(\R^d)^T}(1+\abs{x}^q)\frac{\E[f^2_{\sigma/\sqrt{2}}(x-X)]}{n} dx\right)^{1/2}\\
            &\le C\left(\frac{1+\sigma^{q/2}}{\sigma^{dT/2}}\right)\frac{1}{\sqrt{n}}.
        \end{align}
        for some $C>0$ that depends only on $d, T, q, M_q(\mu)$. We refer to the proof of \cite[Theorem 3.1]{hou2024convergence} for computational details. This proves the desired result.
    \end{proof}

\subsection{Proofs for Section \ref{sec:bandwidth}}\label{sec:pf_bandwidth}

In Lemma \ref{lem:epcoupling}, we show that there is a Monge mapping from $\mu^{\sigma}$ to $\mu$ whose kernels are almost $\w_p$-optimal. We use this to prove Theorem \ref{thm:muandmusigma}. We then present proofs of Corollaries \ref{cor:uniform_moc} and \ref{cor:empiricalvaryingsigma}.

    \begin{lem}\label{lem:epcoupling}
        Let $\mu\in \sp_p((\R^d)^T)$ and $\sigma>0$. For each $\ep>0$ there exists a bicausal coupling $\pi^{\ep}\in \cplbc(\mu^{\sigma}, \mu)$ given by $\pi^{\ep}=(\id, G^{\ep})_{\#}\mu^{\sigma}$, where $G^{\ep}=(G^{\ep}_1, \ldots, G^{\ep}_T):(\R^d)^T\to (\R^d)^T$ is a Borel measurable function, and $\pi^{\ep}$ satisfies
        \begin{align}
            \left(\int \abs{x_1-y_1}^p \pi^{\ep}(dx_1, dy_1)\right)^{1/p}
            \le (1+\ep)\w_p((\mu^{\sigma})_1, \mu_1)
        \end{align}
        as well as
        \begin{align}
            \left(\int \abs{x_{t+1}-y_{t+1}}^p \pi^{\ep}_{x_{1:t},y_{1:t}}(dx_{t+1}, dy_{t+1})\right)^{1/p}
            \le (1+\ep)\w_p((\mu^{\sigma})_{x_{1:t}}, \mu_{y_{1:t}})
        \end{align}
        for all $t\in \{1,2, \ldots, T-1\}$ and $x_{1:t}, y_{1:t}\in (\R^d)^t$. Moreover, $G^{\ep}_t$ depends only on the first $t$ coordinates, i.e. $G^{\ep}_t(x_{1:T})=G^{\ep}_t(x_{1:t})$ for $t\in\{1,2, \ldots, T-1\}$.
    \end{lem}
    \begin{proof}
        Throughout the proof, $\ep>0$ is fixed. Since $(\mu^{\sigma})_1$ is absolutely continuous, it can be approximated by Monge mappings. Precisely speaking, from \cite[Theorem B]{AIHPB_2007__43_1_1_0},
        \begin{align}
            \w_p((\mu^{\sigma})_1, \mu_1)^p
            =\inf_{T: T_{\#}(\mu^{\sigma})_1=\mu_1}\int \abs{x_1-T(x_1)}^p \mu^{\sigma}(dx_1).
        \end{align}
        Hence there exists a map $F^{\ep}_1:\R^d\to \R^d$ such that $(F^{\ep}_1)_{\#}(\mu^{\sigma})_1=\mu_1$ and
        \begin{align}
          \int \abs{x_1-F^{\ep}_1(x_1)}^p \mu^{\sigma}(dx_1)\le (1+\ep)^p\w_p((\mu^{\sigma})_1, \mu_1)^p.
        \end{align}
        Let us define $\pi^{\ep}_{1,1}:=(\id, F^{\ep}_1)_{\#}(\mu^{\sigma})_1\in \cpl((\mu^{\sigma})_1, \mu_1)$. For $t\in \{1,2, \ldots, T-1\}$, let us denote by $\pi^{\star}_{x_{1:t}, y_{1:t}}\in \cpl((\mu^{\sigma})_{x_{1:t}}, \mu_{y_{1:t}})$ a $\w_p$-optimal coupling for each $x_{1:t}, y_{1:t}\in (\R^d)^t$. It can be chosen in such a way that
        \begin{align}
            (\R^d)^t\times (\R^d)^t \ni (x_{1:t}, y_{1:t})
            \mapsto \pi^{\star}_{x_{1:t}, y_{1:t}}\in \sp(\R^d\times \R^d)\label{eq:wppi}
        \end{align}
        is Borel measurable, where $\sp(\R^d \times \R^d)$ is endowed with the weak topology, see e.g. \cite[Corollary 5.22]{villani2021topics}. Thanks to Lemma \ref{lem:monge}, the map \eqref{eq:wppi} can be approximated by parameterized Monge couplings. Applying Lemma \ref{lem:monge} with $\mathcal{X}=\R^d$ and $\mathcal{Z}=(\R^d)^t\times (\R^d)^t$, we find a measurable map $F^{\ep}_{t+1}:(\R^d)^t\times (\R^d)^t\times \R^d\to \R^d$ such that
        \begin{align}
        \int \abs{x_{t+1}-F^{\ep}_{t+1}(x_{1:t}, y_{1:t}, x_{t+1})}^p (\mu^{\sigma})_{x_{1:t}}(dx_{t+1})
        \le (1+\ep)^p\w_p((\mu^{\sigma})_{x_{1:t}}, \mu_{y_{1:t}} )^p.
        \end{align}
        Set $\pi^{\ep}_{x_{1:t}, y_{1:t}}:=(\id, F^{\ep}_{t+1}(x_{1:t}, y_{1:t}, \cdot\,))_{\#}(\mu^{\sigma})_{x_{1:t}}\in \cpl((\mu^{\sigma})_{x_{1:t}}, \mu_{y_{1:t}})$ and define $\pi^{\ep}\in \cplbc(\mu^{\sigma}, \mu)$ via
        \begin{align}
            \pi^{\ep}(dx_{1:T}, dy_{1:T})
            =\pi^{\ep}_{x_{1:T-1}, y_{1:T-1}}(dx_T, dy_T)\cdots \pi^{\ep}_{x_{1}, y_{1}}(dx_2, dy_2)\pi^{\ep}_{1,1}(dx_1, dy_1).
        \end{align}
        Note that if $X\sim \mu^{\sigma}$, then $Y\sim \mu$, where $Y$ is defined via the following recursive relation,
        \begin{align}
            Y_1=F^{\ep}_{1}(X_{1}) \text{ and } Y_{t+1}=F^{\ep}_{t+1}(X_{1:t}, Y_{1:t}, X_{t+1}) \text{ for } t\in\{1,2, \ldots, T-1\}.\label{eq:G^ep}
        \end{align}
        Substituting \eqref{eq:G^ep} iteratively, the construction of $G^{\ep}$ is now straightforward.
    \end{proof}

    \begin{proof}[Proof of Theorem \ref{thm:muandmusigma}]
        Let $\ep>0$ and choose $\pi^{\ep}=(\id, G^{\ep})_{\#}\mu^{\sigma}\in \cplbc(\mu^{\sigma}, \mu)$ as in Lemma \ref{lem:epcoupling}. Note that
        \begin{align}
            \a\w_p(\mu^{\sigma}, \mu)
            \le \left(\int \sum_{t=1}^T \abs{x_t-y_t}^p \pi^{\ep}(dx,dy)\right)^{1/p}
            \le \sum_{t=1}^{T}\left(\int \abs{x_t-y_t}^p \pi^{\ep}(dx, dy)\right)^{1/p}.\label{eq:awbyIt}
        \end{align}
        Using the fact that $\w_p((\mu^{\sigma})_1, \mu_1)\le (M_p(\xi_{\sigma}))^{1/p}=\sigma (M_p(\xi))^{1/p}$, we deduce from Lemma \ref{lem:epcoupling} that
        \begin{align}
            \text{I}^{\ep}_1
            :=\left(\int \abs{x_{1}-y_{1}}^p \pi^{\ep}(dx, dy)\right)^{1/p}
            \le (1+\ep)\w_p((\mu^\sigma)_1, \mu_1)
            \le (1+\ep)\sigma(M_p(\xi))^{1/p}.\label{eq:I1}
        \end{align}
        Now, fix $t\in \{1,2, \ldots, T-1\}$. Since $G^{\ep}_t$ depends only on the first $t$ coordinates, we denote $(G_1^{\ep}(x), \ldots, G_t^{\ep}(x))$ by $G^{\ep}(x_{1:t})$. Also, let us denote the density of $(\xi_{\sigma})_{1:t}$ by
        \begin{align}
            (f_{\sigma})_{1:t}(x_{1:t})
            :=\int_{(\R^d)^{T-t}} f_{\sigma}(x_{1:T})dx_{t+1:T}.
        \end{align}
        Using the fact that $(\xi_{\sigma})_{z_{1:t}}$ has the density
        \begin{align}
            z_{t+1}\mapsto \frac{(f_{\sigma})_{1:t+1}(z_{1:t+1})}{(f_{\sigma})_{1:t}(z_{1:t})},
        \end{align}
        we compute that for measurable $B\subseteq \R^d$,
        \begin{align}
            (\mu^{\sigma})_{x_{1:t}}(B)
            =\frac{\int (f_{\sigma})_{1:t}(x_{1:t}-z_{1:t}) (\mu_{z_{1:t}}\ast (\xi_{\sigma})_{x_{1:t}-z_{1:t}})(B) \mu(dz_{1:t})}{(f_{\sigma})_{1:t}\ast \mu_{1:t}(x_{1:t})}.\label{pfeq:convexity}
        \end{align}
        A similar identity was used in  \cite[proof of Theorem 2.12]{hou2024convergence}. From Lemma \ref{lem:epcoupling} and convexity of $\w_p^p$ applied to \eqref{pfeq:convexity}, we obtain
        \begin{align}
            &\int \abs{x_{t+1}-y_{t+1}}^p \pi^{\ep}(dx, dy)\\
            &= \int \abs{x_{t+1}-y_{t+1}}^p \pi^{\ep}_{x_{1:t}, y_{1:t}} (dx_{t+1}, dy_{t+1}) \pi^{\ep}(dx_{1:t}, dy_{1:t})\\
            &\le \int  (1+\ep)^p \w_p((\mu^\sigma)_{x_{1:t}}, \mu_{y_{1:t}})^p \pi^{\ep}(dx_{1:t}, dy_{1:t})\\
            &= (1+\ep)^p\int \w_p((\mu^{\sigma})_{x_{1:t}}, \mu_{G^{\ep}(x_{1:t})})^p \mu^{\sigma}(dx_{1:t})\\
            &\le (1+\ep)^p \int \frac{\int (f_{\sigma})_{1:t}(x_{1:t}-z_{1:t})\w_p(\mu_{z_{1:t}}\ast (\xi_{\sigma})_{x_{1:t}-z_{1:t}}, \mu_{G^{\ep}(x_{1:t})})^p\mu (dz_{1:t})}{(f_{\sigma})_{1:t}\ast \mu_{1:t}(x_{1:t})} \mu^{\sigma}(dx_{1:t}).
        \end{align}
        Since $x_{1:t}\mapsto (f_{\sigma})_{1:t}\ast \mu_{1:t}(x_{1:t})$ is the density of $(\mu^{\sigma})_{1:t}$, we use the change of variable $w_{1:t}=x_{1:t}-z_{1:t}$ to obtain
        \begin{align}
            &\int \frac{\int (f_{\sigma})_{1:t}(x_{1:t}-z_{1:t})\w_p(\mu_{z_{1:t}}\ast (\xi_{\sigma})_{x_{1:t}-z_{1:t}}, \mu_{G^{\ep}(x_{1:t})})^p\mu(dz_{1:t})}{(f_{\sigma})_{1:t}\ast \mu_{1:t}(x_{1:t})} \mu^{\sigma}(dx_{1:t})\\
            &=\int \int (f_{\sigma})_{1:t}(x_{1:t}-z_{1:t})\w_p(\mu_{z_{1:t}}\ast (\xi_{\sigma})_{x_{1:t}-z_{1:t}}, \mu_{G^{\ep}(x_{1:t})})^p \mu(dz_{1:t}) dx_{1:t} \\
            &=\int \int \w_p(\mu_{z_{1:t}}\ast (\xi_{\sigma})_{w_{1:t}}, \mu_{G^{\ep}(z_{1:t}+w_{1:t})})^p (\mu_{1:t}\otimes (\xi_{\sigma})_{1:t})(dz_{1:t}, dw_{1:t})\\
            &=\E[\w_p(\mu_{X_{1:t}}\ast (\xi_{\sigma})_{\sigma Z_{1:t}}, \mu_{G^{\ep}(X_{1:t}+\sigma Z_{1:t})})^p],
        \end{align}
        where $(X, Z)\sim \mu\otimes \xi$. Using the above, we bound
        \begin{align}
            &\text{I}^{\ep}_{t+1}:=\left(\int \abs{x_{t+1}-y_{t+1}}^p \pi^{\ep}(dx, dy)\right)^{1/p}\\
            &\le (1+\ep) \E[\w_p(\mu_{X_{1:t}}\ast (\xi_{\sigma})_{\sigma Z_{1:t}}, \mu_{G^{\ep}(X_{1:t}+\sigma Z_{1:t})})^p]^{1/p}\\
            &\le (1+\ep) (\E[\w_p(\mu_{X_{1:t}}\ast (\xi_{\sigma})_{\sigma Z_{1:t}}, \mu_{X_{1:t}})^p]^{1/p}+ \E[\w_p(\mu_{X_{1:t}}, \mu_{G^{\ep}(X_{1:t}+\sigma Z_{1:t})})^p]^{1/p})\\
            &\le (1+\ep)\left(\sigma (M_p(\xi))^{1/p}+\E[\w_p(\mu_{X_{1:t}}, \mu_{G^{\ep}(X_{1:t}+\sigma Z_{1:t})})^p]^{1/p}\right)\\
            &\le (1+\ep)\left(\sigma (M_p(\xi))^{1/p}+\omega_{\mu}^{t,p}(\norm{X_{1:t}-G^{\ep}(X_{1:t}+\sigma Z_{1:t})}_{L^p})\right)\\
            &\le (1+\ep)\left(\sigma (M_p(\xi))^{1/p}+\omega_{\mu}^{t,p}(\sigma (M_p(\xi))^{1/p}+\norm{X_{1:t}+\sigma Z_{1:t}-G^{\ep}(X_{1:t}+\sigma Z_{1:t})}_{L^p})\right)\\
            &\le (1+\ep)\left(\sigma (M_p(\xi))^{1/p}+\omega_{\mu}^{t,p}\left(\sigma (M_p(\xi))^{1/p}+\sum_{s=1}^{t}I^{\ep}_{s}\right)\right).\label{eq:It}
        \end{align}
        Here, the third inequality is immediate from
        \begin{align}
            \E[\w_p(\mu_{X_{1:t}}\ast (\xi_{\sigma})_{\sigma Z_{1:t}}, \mu_{X_{1:t}})^p]
            &\le \E\left[\int \abs{z_{t+1}}^p (\xi_{\sigma})_{\sigma Z_{1:t}}(dz_{t+1})\right]\\
            &=\E\left[\int \abs{z_{t+1}}^p \frac{(f_{\sigma})_{1:t+1}(\sigma Z_{1:t}, z_{t+1})}{(f_{\sigma})_{1:t}(\sigma Z_{1:t})}dz_{t+1}\right]\\
            &=\int\int \abs{z_{t+1}}^p \frac{(f_{\sigma})_{1:t+1}(z_{1:t+1})}{(f_{\sigma})_{1:t}(z_{1:t})}dz_{t+1}(f_{\sigma})_{1:t}(z_{1:t})dz_{1:t}\\
            &=\sigma^p M_p(\xi_{t+1}).
        \end{align}

        The definition of $\omega_{\mu}^{t,p}$ justifies the fourth inequality. The fifth inequality is the triangle inequality and monotonicity of $\omega_{\mu}^{t,p}$. To see the last inequality, note that $(X+\sigma Z, G^{\ep}(X+\sigma Z))\sim \pi^{\ep}$ and
        \begin{align}
            \norm{X_{1:t}+\sigma Z_{1:t}-G^{\ep}(X_{1:t}+\sigma Z_{1:t})}_{L^p}
            =\left(\int \abs{x_{1:t}-y_{1:t}}^p \pi^{\ep}(dx_{1:t}, dy_{1:t})\right)^{1/p}
            \le \sum_{s=1}^t \text{I}^{\ep}_s.
        \end{align}
        Starting from the initial estimate \eqref{eq:I1}, we bound $I^{\ep}_t$ by $h^{t-1, p}_{\mu}$ using the recursive estimate $\eqref{eq:It}$ along with Proposition~\ref{prop:propertyofomega}\ref{prop:propertyofomega_const} as follows,
        \begin{align}
            &\liminf_{\ep \to 0}I^{\ep}_{t}
            \le (1\vee (M_p(\xi))^{1/p})(h^{0,p}_{\mu}(\sigma)+th^{t-1,p}_{\mu}(\sigma))\label{pfeq:constantinduction}
        \end{align}
        for all $t\in \{2, 3,\ldots, T\}$. This can be shown by the induction. When $t=2$, we compute from \eqref{eq:I1} and \eqref{eq:It} that 
        \begin{align}
            \liminf_{\ep\to 0}I^{\ep}_2
            &\le \sigma (M_p(\xi))^{1/p}+\omega_{\mu}^{1,p}\left(\sigma (M_p(\xi))^{1/p}+I^{\ep}_{1}\right)\\
            &\le \sigma (M_p(\xi))^{1/p}+\omega_{\mu}^{1,p}\left(2\sigma (M_p(\xi))^{1/p}\right)
            \le (1\vee (M_p(\xi))^{1/p})(\sigma+2\omega^{1,p}_{\mu}(\sigma)).
        \end{align}
        Here, the last inequality follows from Proposition~\ref{prop:propertyofomega}\ref{prop:propertyofomega_const}. This proves \eqref{pfeq:constantinduction} when $t=2$. Assuming \eqref{pfeq:constantinduction} holds for $I^{\ep}_s$, $s\in \{2,\ldots, t\}$, a similar computation shows that
        \begin{align}
            \liminf_{\ep\to 0} I^{\ep}_{t+1}
            &\le \sigma(M_p(\xi))^{1/p}+\omega^{t,p}_{\mu}\left(\sigma(M_p(\xi))^{1/p}+I^{\ep}_1+\sum_{s=2}^{t}I^{\ep}_{s}\right)\\
            &\le \sigma(M_p(\xi))^{1/p}
            +\omega^{t,p}_{\mu}\left(2\sigma(M_p(\xi))^{1/p}+\sum_{s=2}^{t}(1\vee (M_p(\xi))^{1/p})(h^{0,p}_{\mu}(\sigma)+sh^{s-1,p}_{\mu}(\sigma))\right)\\
            &\le \sigma(M_p(\xi))^{1/p}
            +(t+1)(1\vee (M_p(\xi))^{1/p})\omega^{t,p}_{\mu}\left(\sum_{s=0}^{t-1}h^{s,p}_{\mu}(\sigma)\right).
        \end{align}
        Hence \eqref{pfeq:constantinduction} holds for all $t\in \{2,3,\ldots, T\}$.
        Plugging this back into \eqref{eq:awbyIt},
        \begin{align}
            \a\w_p(\mu^{\sigma}, \mu)
            \le \liminf_{\ep \to 0}\sum_{t=1}^{T}I^{\ep}_{t}
            \le \sigma(M_p(\xi))^{1/p}+\sum_{t=2}^{T}(1\vee (M_p(\xi))^{1/p})(h^{0,p}_{\mu}(\sigma)+th^{t-1,p}_{\mu}(\sigma)).
        \end{align}
        This proves the desired estimate.
    \end{proof}

    \begin{proof}[Proof of Corollary \ref{cor:uniform_moc}]
        \ref{cor:uniform_moc_holder}  From Proposition~\ref{prop:omega_holder}, $\omega^{t,p}_{\mu}(\sigma)\le L \sigma^{\alpha}$ if $\mu\in \sp_{p, \alpha, L}((\R^d)^T)$. Note from the construction of $h^{t,p}_{\mu}$ that $h^{1,p}_{\mu}(\sigma)=\omega^{1,p}_{\mu}(\sigma)\le L\sigma^{\alpha}$ and
        \begin{align}
            h^{2,p}_{\mu}(\sigma)
            =\omega^{2,p}_{\mu}(\sigma+h^{1,p}_{\mu}(\sigma))
            \le L(\sigma+L\sigma^{\alpha})^{\alpha}
            \le L(\sigma^{\alpha}+L^{\alpha}\sigma^{\alpha^2})
            \le (L\vee L^{\alpha+1})(\sigma^{\alpha}+\sigma^{\alpha^2}).
        \end{align}
        It is straightforward from the mathematical induction that $h^{t,p}_{\mu}(\sigma)\le C (\sigma^{\alpha}+\sigma^{\alpha^2}+\ldots+\sigma^{\alpha^{t}})$ for some $C>0$ that depends only on $L, \alpha$ if $t\in \{1,2,\ldots, T-1\}$. Applying this bound to Theorem \ref{thm:muandmusigma} yields the desired estimate.

        \ref{cor:uniform_moc_general}  From Theorem \ref{thm:muandmusigma}, we find
        \begin{align}
            \sup_{\mu\in K}\a\w_p(\mu^{\sigma}, \mu)
            \le T(1\vee (M_p(\xi))^{1/p})\sum_{t=0}^{T-1}h^{t,p}_{K}(\sigma)\label{eq:uniform_bandwidth}
        \end{align}
        where we define $h^{0,p}_{K}(\sigma):=\sigma$ and $h^{t,p}_{K}(\sigma):=\sup_{\mu\in K} \omega^{t,p}_{\mu}\left(\sum_{s=0}^{t-1} h^{s,p}_{K}(\sigma)\right)$, $t\in \{1,2,\ldots, T-1\}$. Note that if $\sup_{\mu\in K}\omega^{t,p}_{\mu}(\delta)\to 0$ as $\delta\to 0$, the right hand side of \eqref{eq:uniform_bandwidth} goes to $0$ as $\sigma\to 0$. The last statement is a straightforward application of a triangle inequality.
    \end{proof}

    \begin{proof}[Proof of Corollary \ref{cor:empiricalvaryingsigma}]
        From Corollaries~\ref{cor:slowrate} and \ref{cor:uniform_moc}\ref{cor:uniform_moc_holder}, we estimate
        \begin{align}
            \E[\a\w_p(\mu, (\boldsymbol{\mu_n})^{\sigma})^p]
            &\le C(\E[\a\w^{(\sigma)}_p(\mu, \boldsymbol{\mu_n})^p]+\E[\a\w_p(\mu^{\sigma}, \mu)^p])\\
            &\le C\left(\frac{1+\sigma^{(q+p)/2}}{\sigma^{(1+dT/2)(1-p/q)}}\right)n^{-(q-p)/(2q)}
            +C\sum_{t=0}^{T-1}\sigma^{p\alpha^{t}}.
        \end{align}
        We optimize the convergence rate of the right hand side by choosing $\sigma = n^{-\gamma}$ where
        \begin{align}
            \gamma=\frac{1-p/q}{(dT+2)(1-p/q)+2p\alpha^{T-1}}.
        \end{align}
    \end{proof}

\subsection{Proofs for Section \ref{sec:boundingAW}}\label{sec:pf_boundingAW}
      \begin{proof}[Proof of Proposition \ref{prop:boundingAWvanilla}]
          This is straightforward from the inequality
          \begin{align}
              \a\w_p(\mu, \nu)^p
              \le 3^{p-1}\a\w^{(\sigma)}_p(\mu, \nu)
              +3^{p-1}\a\w_p(\mu, \mu^{\sigma})
              +3^{p-1}\a\w_p(\nu, \nu^{\sigma})
          \end{align}
          applied to Theorem \ref{thm:awsigmaw1} and Theorem \ref{thm:muandmusigma}.
      \end{proof}

      \begin{proof}[Proof of Corollary \ref{cor:boundingAW}]
          Choose $\xi$ to be supported on $\{x\in (\R^d)^T :\, \abs{x}\le 1\}$. From $0<\sigma<R$, we have
          \begin{align}
              \int_{\{\abs{x}\ge 2R\}}\abs{x}^p \mu^{\sigma}(dx)
              &=\int_{\{\abs{x+\sigma z}\ge 2R\}}\abs{x+\sigma z}^p \mu(dx)\xi(dz)\\
              &\le \int_{\{\abs{x}\ge R\}}\abs{x+\sigma z}^p \mu(dx)\xi(dz)
              \le 2^{p-1}\int_{\{\abs{x}\ge R\}}\abs{x}^p \mu(dx)
              +2^{p-1}\sigma^p.
          \end{align}
          Here, the last inequality follows from
          \begin{align}
              \abs{x+\sigma z}^p
              \le 2^{p-1}(\abs{x}^p+\sigma^p \abs{z}^p)
              \le 2^{p-1}(\abs{x}^p+\sigma^p)
              \text{ for }\mu\otimes \xi\text{-a.e. }(x,z).
          \end{align}
          Recalling from Theorem~\ref{thm:muandmusigma} that $h^{0,p}_{\mu}(\sigma)=\sigma$, we obtain
          \begin{align}
              \int_{\{\abs{x}\ge 2R\}}\abs{x}^p \mu^{\sigma}(dx)
              \le 2^{p-1}\int_{\{\abs{x}\ge R\}}\abs{x}^p \mu(dx)
              +2^{p-1}\left(\sum_{t=0}^{T-1}h^{t,p}_{\mu}(\sigma)\right)^p.
          \end{align}
          Applying this bound to Proposition \ref{prop:boundingAWvanilla} yields the desired result.
      \end{proof}
      
      \begin{proof}[Proof of Corollary \ref{cor:boundingAWcpt}]
      \ref{cor:boundingAWcpt_vanilla} By a translation argument we may assume that $0\in F$. Let $\mu, \nu\in \sp(F)$. Choose $\xi$ to be supported on $\{x\in (\R^d)^T :\, \abs{x}\le 1\}$. By combining the triangle inequality with Corollary \ref{cor:awsigmaw1_mom}\ref{cor:awsigmaw1_mom_cpt} and Theorem \ref{thm:muandmusigma}, we compute
      \begin{align}
          \a\w_p(\mu, \nu)
          &\le \a\w^{(\sigma)}_p(\mu, \nu)
          +2\left(\a\w_p(\mu, \mu^{\sigma})\vee \a\w_p(\nu, \nu^{\sigma})\right)\\
          &\le C\left((1+\sigma)\left(\frac{\w_1(\mu, \nu)}{\sigma}\right)^{1/p}
          +\sum_{t=0}^{T-1}h^{t,p}_{\mu}(\sigma)\vee h^{t,p}_{\nu}(\sigma)\right)
      \end{align}
      for some $C>0$ that depends only on $d, T, \diam(F)$. This proves the desired estimate.

      \ref{cor:boundingAWcpt_holder} Let $\mu, \nu\in \sp(F)\cap \sp_{p, \alpha, L}((\R^d)^T)$. Using Corollary~\ref{cor:boundingAWcpt}\ref{cor:boundingAWcpt_vanilla} and Corollary \ref{cor:uniform_moc}\ref{cor:uniform_moc_holder}, we estimate
        \begin{align}
            \a\w_p(\mu, \nu)
            \le C\left( (1+\sigma)\left(\frac{\w_1(\mu,\nu)}{\sigma}\right)^{1/p}
            +\sum_{t=0}^{T-1}\sigma^{\alpha^t}\right)\label{pfeq:awrelcpt_w}
        \end{align}
        with some constant $C>0$ that depends only on $d, T, L, \alpha, \diam(F)$. In \eqref{pfeq:awrelcpt_w}, choose $\sigma=\w_1(\mu, \nu)^{\beta}$ for $\beta=(p\alpha^{T-1}+1)^{-1}$ and we bound the right-hand side by
        \begin{align}
            &(1+\sigma)\left(\frac{\w_1(\mu,\nu)}{\sigma}\right)^{1/p}
            +\sum_{t=0}^{T-1}\sigma^{\alpha^t}\\
            &=(1+\w_1(\mu, \nu)^{\beta})\w_1(\mu,\nu)^{\alpha^{T-1}/(p\alpha^{T-1}+1)}
            +\sum_{t=0}^{T-1}\w_1(\mu,\nu)^{\alpha^{t}/(p\alpha^{T-1}+1)}\\
            &\le \left(1+(2\diam(F))^{\beta}+\sum_{t=0}^{T-1}(2\diam(F))^{(\alpha^t-\alpha^{T-1})/(p\alpha^{T-1}+1)}\right)\w_1(\mu,\nu)^{\alpha^{T-1}/(p\alpha^{T-1}+1)}.
        \end{align}
        Here, we make use of the fact that $\w_1(\mu, \nu)\le M_1(\mu)+M_1(\nu)\le 2\diam(F)$. This proves the desired result.
    \end{proof}

\subsection{Proofs for Section \ref{sec:topology}}\label{sec:pf_topology}
In this section we prove Theorem \ref{thm:slowsigma}, Theorem \ref{thm:fastsigma} and Theorem \ref{thm:fixedsigma}. All the proofs are based on the following simple reformulation of the sequential convergence.

\begin{lem}\label{lem:subseq}
    Let $\x$ be a metric space. Suppose $(x_n)_{n\in \N}\subseteq \x$ and $x\in \x$. Then $x_n\to x$ in $\x$ if and only if for every subsequence $(x_{n_k})_{k\in \N}$ of $(x_n)_{n\in \N}$, there exists a subsequence of $(x_{n_k})_{k\in \N}$ that converges to $x$.
\end{lem}

\begin{proof}[Proof of Theorem \ref{thm:slowsigma}]
    \ref{thm:slowsigma_imply} Suppose $\a\w^{(\sigma_n)}_p(\mu_n, \mu)\to 0$. From Lemma \ref{lem:subseq}, by passing to a subsequence, it is sufficient to assume that $\sigma_n\to \sigma_0$ for some $\sigma_0\in [0,\infty)$. From Proposition~\ref{prop:smoothW},
    \begin{align}
        \w^{(\sigma_0)}_p(\mu_{n}, \mu)
        \le \w_p^{(\sigma_n)}(\mu_n, \mu)
        +2\abs{\sigma_n-\sigma_0}(M_p(\xi))^{1/p}
        \le \a\w_p^{(\sigma_n)}(\mu_n, \mu)
        +2\abs{\sigma_n-\sigma_0}(M_p(\xi))^{1/p}.
    \end{align}
    Since the right hand side converges to $0$ as $n\to \infty$, we conclude that $\mu_n\to \mu$ in $\w^{(\sigma_0)}_p$. We conclude that $\mu_n\to \mu$ in $\w_p$ thanks to Proposition~\ref{prop:smoothW}.

    \ref{thm:slowsigma_reverse} Suppose now that $\w_p(\mu_n, \mu)\to 0$ and $\w_1(\mu_n, \mu)/\sigma_n\to 0$. From Lemma \ref{lem:subseq}, we may assume that $\sigma_n\to \sigma_0$ for some $\sigma_0\in [0,\infty)$. We find from Theorem \ref{thm:awsigmaw1} that for any $R>0$,
        \begin{align}
            \limsup_{n\to \infty}\a\w^{(\sigma_n)}_p(\mu_n, \mu)
            \le 6^{p-1}T2^T \sup_{n\in \N}\int_{\{\abs{x}\ge R\}} \abs{x}^p (\mu_n^{\sigma_n}+\mu^{\sigma_n})(dx).\label{eq:proof_w_prelcpt}
        \end{align}
        Note that $\mu^{\sigma_n}_n\to \mu^{\sigma_0}$ and $\mu^{\sigma_n}\to \mu^{\sigma_0}$ in $\w_p$. In particular, $\{\mu^{\sigma_n}_n, \mu^{\sigma_n} : n\in \N\}$ is $\w_p$-relatively compact. Thus, the right hand side of \eqref{eq:proof_w_prelcpt} can be made arbitrarily small as $R\to \infty$. This shows that $\a\w^{(\sigma_n)}_p(\mu_n, \mu)\to 0$.
\end{proof}

\begin{proof}[Proof of Theorem \ref{thm:fastsigma}]
    \ref{thm:fastsigma_imply}  Suppose $\a\w_p(\mu_n, \mu)\to 0$. As in the proof of Theorem \ref{thm:slowsigma}, it is sufficient to prove $\a\w^{(\sigma_n)}_p(\mu_n, \mu)\to $ when $\sigma_n\to \sigma_0$ for some $\sigma_0\in [0,\infty)$. If $\sigma_0>0$, then $\w_1(\mu_n, \mu)/\sigma_n\to 0$. By Theorem~\ref{thm:slowsigma}\ref{thm:slowsigma_reverse} this implies that $\a\w^{(\sigma_n)}_p(\mu_n, \mu)\to 0$. Suppose now that $\sigma_0=0$. By the triangle inequality,
    \begin{align}
        \a\w^{(\sigma_n)}_p(\mu_n, \mu)
        \le \a\w_p(\mu_n, \mu)
        +2\sup_{\nu\in K}\a\w_p(\nu^{\sigma_n}, \nu)
    \end{align}
    where $K=\{\mu, \mu_n :\, n\in \N\}$. Note that $K$ is $\a\w_p$-relatively compact. Thus, Proposition \ref{prop:awrelcpt} and Corollary~\ref{cor:uniform_moc}\ref{cor:uniform_moc_general} show that $\sup_{\nu\in K}\a\w_p(\nu^{\sigma_n}, \nu)\to 0$. In conclusion, $\a\w^{(\sigma_n)}_p(\mu_n, \mu)\to 0$.

    \ref{thm:fastsigma_reverse} Suppose $\a\w^{(\sigma_n)}_p(\mu_n, \mu)\to 0$ and $h^{t,p}_{\mu_n}(\sigma_n)\to 0$ for all $t\in \{1,2,\ldots, T-1\}$. Since $\omega^{t,p}_{\mu_n}\le h^{t,p}_{\mu_n}$, we have $\omega^{t,p}_{\mu_n}(\sigma_n)\to 0$. Using Lemma \ref{lem:subseq}, we may assume that $\sigma_n\to \sigma_0$ for some $\sigma_0\in [0,\infty)$. If $\sigma_0>0$, fix $\ep>0$ and choose $N\in \N$ such that $\sup_{n\ge N}\omega^{t,p}_{\mu_n}(\sigma_n)\le \ep$ and $\inf_{n\ge N}\sigma_n\ge\sigma_0/2$. Note that
    \begin{align}
        \limsup_{\delta\to 0}\sup_{n\in \N}\omega^{t,p}_{\mu_n}(\delta)
        \le \limsup_{\delta\to 0} \sum_{n=1}^{N}\omega^{t,p}_{\mu_n}(\delta)
        +\sup_{n\ge N}\omega^{t,p}_{\mu_n}(\sigma_0/2)
        \le 0+\sup_{n\ge N}\omega^{t,p}_{\mu_n}(\sigma_n)
        \le \ep.
    \end{align}
    Since $\ep>0$ was arbitrary and $\w_p(\mu_n, \mu)\to 0$, we deduce from Proposition \ref{prop:awrelcpt} that $\{\mu_n :\, n\in \N\}$ is $\a\w_p$-relatively compact. Thus, $\a\w_p(\mu_n, \mu)\to 0$.

    Now, suppose $\sigma_0=0$. Then by the triangle inequality,
    \begin{align}
        \limsup_{n\to \infty}\a\w_p(\mu_n, \mu)
        &\le \limsup_{n\to \infty}\left(\a\w^{(\sigma_n)}_p(\mu_n, \mu)
        +\a\w_p(\mu_n^{\sigma_n}, \mu_n)
        +\a\w_p(\mu^{\sigma_n}, \mu)\right)\\
        &=\limsup_{n\to \infty}\a\w_p(\mu_n^{\sigma_n}, \mu_n).
    \end{align}
    Here, the equality follows from Corollary \ref{cor:uniform_moc}\ref{cor:uniform_moc_general} applied to $K=\{\mu\}$. It is easy to check from Theorem \ref{thm:muandmusigma} that $h^{t,p}_{\mu_n}(\sigma_n)\to 0$ for all $t\in \{1,2,\ldots, T-1\}$ implies that $\a\w_p(\mu^{\sigma_n}_n, \mu_n)\to 0$. Hence, we conclude that $\a\w_p(\mu_n, \mu)\to 0$.
\end{proof}

\begin{proof}[Proof of Theorem \ref{thm:fixedsigma}]
    From Proposition~\ref{prop:smoothW} and the fact that $\w^{(\sigma)}_p\le C\a\w^{(\sigma)}_p$, it is obvious that if $\mu_n\to \mu$ in $\a\w^{(\sigma)}_p$, then $\mu_n\to \mu$ in $\w_p$. Conversely, let us assume $\w_p(\mu_n, \mu)\to 0$. By choosing $\sigma_n\equiv\sigma$ in Theorem~\ref{thm:slowsigma}\ref{thm:slowsigma_reverse}, we obtain that $\a\w^{(\sigma)}_p(\mu_n, \mu)\to 0$. Therefore, $\a\w^{(\sigma)}_p$ metrizes the $\w_p$-topology. 
\end{proof}

\appendix

\section{Proof regarding Example \ref{ex:standardex}}\label{appendix:pfex}
    \begin{ex}
        Let $(\ep_n)_{n\in \N}, (\sigma_n)_{n\in \N}\subseteq (0, \infty)$ such that $\ep_n\to 0$ and $\sup_{n\in \N}\sigma_n<\infty$. We define $\mu_n:=\frac{1}{2}\delta_{(\ep_n, 1)}+\frac{1}{2}\delta_{(-\ep_n, -1)}$, $\mu:=\frac{1}{2}\delta_{(0,1)}+\frac{1}{2}\delta_{(0,-1)}$, and $\xi_{\sigma_n}:=\n(0, \sigma^2_n\Id_2)$. Then we have
        \begin{anumerate}
            \item\label{ex:slow} $\frac{\ep_n}{\sigma_n}\to 0$ if and only if $\frac{\w_1(\mu_n, \mu)}{\sigma_n}\to 0$ if and only if $\a\w^{(\sigma_n)}_p(\mu_n, \mu)\to 0$.
            \item\label{ex:fast} $\frac{\ep_n}{\sigma_n}\to \infty$ if and only if $\omega^{1,p}_{\mu_n}(\sigma_n)\to 0$ if and only if $\a\w_p(\mu_n^{\sigma_n}, \mu_n)\to 0$.
        \end{anumerate}
    \end{ex}

    \begin{proof}
        \ref{ex:slow} It is evident from $\w_p(\mu_n, \mu)=\ep_n$ that $\ep_n/\sigma_n\to 0$ if and only if $\w_1(\mu_n, \mu)/\sigma_n\to 0$. The fact that $\w_1(\mu_n, \mu)/\sigma_n\to 0$ implies $\a\w^{(\sigma_n)}_p(\mu_n, \mu)\to 0$ follows from Theorem~\ref{thm:slowsigma}\ref{thm:slowsigma_reverse}. Let us assume that $\a\w^{(\sigma_n)}_p(\mu_n, \mu)\to 0$. We compute $(\mu^{\sigma_n})_{1}=N(0, \sigma_n^2)$, $(\mu^{\sigma_n})_x=\frac{1}{2}\n(1, \sigma_n^2)+\frac{1}{2}\n(-1, \sigma_n^2)$ and
        \begin{align}
        &(\mu^{\sigma_n}_n)_1=\frac{1}{2}\n(\ep_n, \sigma_n^2)+\frac{1}{2}\n(-\ep_n, \sigma_n^2),\\
        &(\mu^{\sigma_n}_n)_y
        =\frac{\varphi_{\sigma_n}(y-\ep_n)}{\varphi_{\sigma_n}(y+\ep_n)+\varphi_{\sigma_n}(y-\ep_n)}\n(1,\sigma_n^2)+\frac{\varphi_{\sigma_n}(y+\ep_n)}{\varphi_{\sigma_n}(y+\ep_n)+\varphi_{\sigma_n}(y-\ep_n)}\n(-1, \sigma^2_n)
        \end{align}
        where $\varphi(x):=(2\pi)^{-1/2}e^{-x^2/2}$ and $\varphi_{\sigma_n}(\cdot):=(\sigma_n)^{-1}\varphi(\cdot/\sigma_n)$. Since $(\mu^{\sigma_n})_x$ does not depend on $x$, 
        \begin{align}
        \a\w^{(\sigma_n)}_p(\mu, \mu_n)^p
        &=\inf_{\gamma\in \cpl((\mu^{\sigma_n})_1, (\mu^{\sigma_n}_n)_1)}\int_{\R^2} \abs{x-y}^p +\w_p((\mu^{\sigma_n})_x, (\mu_n^{\sigma_n})_y)^p \gamma(dx, dy)\\
        &=\w_p((\mu^{\sigma_n})_1, (\mu^{\sigma_n}_n)_1)^p
        +\int_{-\infty}^{\infty} \w_p((\mu^{\sigma_n})_x, (\mu_n^{\sigma_n})_y)^p(\mu^{\sigma_n}_n)_1(dy).
    \end{align}
    Using a change of variable, we write
    \begin{align}
        &\int_{-\infty}^{\infty} \w_p((\mu^{\sigma_n})_x, (\mu_n^{\sigma_n})_y)^p(\mu^{\sigma_n}_n)_1(dy)\\
        &=\int_{-\infty}^{\infty} \w_p\left(\n(1, \sigma^2_n)/2+\n(-1,\sigma^2_n)/2, a_n(y)\n(1, \sigma^2_n)+b_n(y)\n(-1,\sigma^2_n)\right)^p\varphi(y)dy,\label{pfeq:ex_e}
    \end{align}
    where $a_n(y):=\varphi(y)/(\varphi(y)+\varphi(y+2\ep_n/\sigma_n))$ and $b_n(y):=1-a_n(y)$. If $\ep_n/\sigma_n\nrightarrow 0$, then up to a subsequence, we may assume that $\ep_n/\sigma_n\to c$ for some $c\in (0, \infty]$. From \eqref{pfeq:ex_e} and the Fatou lemma,
    \begin{align}
        0=\lim_{n\to \infty}\a\w_p^{(\sigma_n)}(\mu, \mu_n)^p
        &\ge \liminf_{n\to \infty}\int_{-\infty}^{\infty} \w_p((\mu^{\sigma_n})_x, (\mu_n^{\sigma_n})_y)^p(\mu^{\sigma_n}_n)_1(dy)\\
        &= \liminf_{n\to \infty}\int_{-\infty}^{\infty} \frac{1}{2} (\w_p((\mu^{\sigma_n})_x, (\mu_n^{\sigma_n})_{\epsilon_n+\sigma_n y})^p + \w_p((\mu^{\sigma_n})_x, (\mu_n^{\sigma_n})_{-\epsilon_n+\sigma_n y} )^p ) \varphi(y)dy\\
        &\ge\int_{-\infty}^{\infty}\w_p\left(\frac{\delta_1+\delta_{-1}}{2}, \frac{\varphi(y)\delta_1+\varphi(y+2c)\delta_{-1}}{\varphi(y)+\varphi(y+2c)}\right)^p\varphi(y)dy
        >0.
    \end{align}
    Thus, $\a\w^{(\sigma_n)}_p(\mu, \mu_n)\to 0$ only if $\ep_n/\sigma_n\to 0$. 
    
    \ref{ex:fast} It is easy check that $\omega^{1,p}_{\mu_n}(\sigma_n)=(\sigma_n/\ep_n)\wedge 2$. Hence $\ep_n/\sigma_n\to \infty$ if and only if $\omega^{1,p}_{\mu_n}(\sigma_n)\to 0$. From Theorem~\ref{thm:muandmusigma}, we have $\a\w_p(\mu^{\sigma_n}_n, \mu_n)
        \le C (\sigma_n+\omega^{1,p}_{\mu_n}(\sigma_n))$. Hence $\omega^{1,p}_{\mu_n}(\sigma_n)\to 0$ yields $\a\w_p(\mu^{\sigma_n}_n, \mu_n)\to 0$. Now, let us assume that $\a\w_p(\mu^{\sigma_n}_n, \mu_n)\to 0$. Note that
    \begin{align}
        \a\w_p(\mu_n^{\sigma_n}, \mu_n)^p
        =\inf_{\gamma\in \cpl((\mu_n^{\sigma_n})_1, (\mu_n)_1)}
        \int_{\R^2} \abs{x-y}^p +\w_p((\mu_n^{\sigma_n})_x, (\mu_n)_y)^p\gamma(dx, dy).\label{pfeq:ex_awsmooth}
    \end{align}
    Choose $\gamma^{\star}$ to be optimal for \eqref{pfeq:ex_awsmooth}. Let $Z\sim \n(0,1)$, $c_n(x):=\varphi_{\sigma_n}(x-\ep_n)/(\varphi_{\sigma_n}(x-\ep_n)+\varphi_{\sigma_n}(x+\ep_n))$ and $d_n(x):=1-c_n(x)$. Then
    \begin{align}
        \a\w_p(\mu_n^{\sigma_n}, \mu_n)^p
        &\ge\int_{\{y=\ep_n\}} \w_p((\mu_n^{\sigma_n})_x, \delta_1)^p\gamma^{\star}(dx, dy)
        +\int_{\{y=-\ep_n\}}\w_p((\mu_n^{\sigma_n})_x, \delta_{-1})^p\gamma^{\star}(dx, dy)\\
        &=\int_{\{y=\ep_n\}} c_n(x)\E[\abs{\sigma_n Z}^p]+d_n(x)\E[\abs{2+\sigma_n Z}^p] \gamma^{\star}(dx, dy)\\
        &+\int_{\{y=-\ep_n\}} c_n(x)\E[\abs{2+\sigma_n Z}^p]+d_n(x)\E[\abs{\sigma_n Z}^p] \gamma^{\star}(dx, dy)\\
        &\ge (\E[\abs{\sigma_n Z}^p]+\E[\abs{2+\sigma_n Z}^p])\int_{-\infty}^{\infty} (c_n(x)\wedge d_n(x))(\mu^{\sigma_n}_n)_1(dx).
    \end{align}
    We compute
    \begin{align}
        \int_{-\infty}^{\infty} (c_n(x)\wedge d_n(x))(\mu^{\sigma_n}_n)_1(dx)
        &=\frac{1}{2}\int_{-\infty}^{\infty} (\varphi_{\sigma_n}(x-\ep_n)\wedge\varphi_{\sigma_n}(x+\ep_n))dx\\
        &=\frac{1}{2}\left(\int_{-\infty}^{-\ep_n/\sigma_n}+\int_{\ep_n/\sigma_n}^{\infty}\right)\varphi(x)dx.
    \end{align}
    Hence $\a\w_p(\mu^{\sigma_n}_n, \mu_n)\to 0$ implies $\sigma_n/\ep_n\to \infty$.
    \end{proof}

\section{Monge maps}\label{appendix:monge}

    Let $\mathcal{A}, \mathcal{B}, \mathcal{Z}$ be Polish spaces. We say that $\pi$ is a kernel from $\mathcal{Z}$ to $\mathcal{A}$ if $\pi$ is a Borel measurable map from $\mathcal{Z}$ to $\sp(\mathcal{A})$. Here, $\sp(\mathcal{A})$ is endowed with the weak topology. We will denote the probability measure $\pi(z)$ by $\pi^{z}$. We use a similar notation for functions. Given a Borel measurable function $T:\mathcal{Z}\times \mathcal{A}\to \mathcal{B}$, we define $T^z:\mathcal{A}\to \mathcal{B}$ via $T^z(a)=T(z,a)$. For a partition $\mathcal{R}$ of $\mathcal{X}$, we define its mesh via $\norm{\mathcal{R}}=\sup_{A\in \mathcal{R}}\diam(A)$.

    \begin{lem}\label{lem:monge}
        Let $\mathcal{X},\mathcal{Z}$ be Polish spaces and let $\pi$ be a kernel from $\mathcal{Z}$ to $\mathcal{X}\times \mathcal{X}$. Let $\mu^{z}$ and $\nu^{z}$ denote the pushforward measure of $\pi^{z}$ through a projection on to the first coordinate and the second coordinate, respectively, i.e. $\pi^{z}\in \cpl(\mu^{z}, \nu^{z})$. Suppose $\mu^z, \nu^{z}\in \sp_p(\mathcal{X})$ and $\mu^{z}$ has no atoms for all $z\in \mathcal{Z}$. Then there exist Borel measurable maps $T_n:\mathcal{Z}\times \mathcal{X}\to \mathcal{X}$ such that $(\id, T^z_n)_{\#}\mu^z\in \cpl(\mu^z, \nu^z)$ and
        \begin{align}
            \int_{\mathcal{X}\times \mathcal{X}} d_{\mathcal{X}}(x,y)^p \pi^{z}(dx, dy)
            =\lim_{n\to \infty} \int_{\mathcal{X}} d_{\mathcal{X}}(x, T_n^{z}(x))^p \mu^{z}(dx).
        \end{align}
        Here $d_{\mathcal{X}}$ is a metric on $\mathcal{X}$ that induces the Polish topology.
    \end{lem}
    \begin{proof}
        It follows from \cite[Theorem 2.3]{beiglbock2022denseness} that there exists a Borel measurable map $T:\mathcal{Z}\times \mathcal{X}\times [0,1]\to \mathcal{X}\times [0,1]$ such that
        \begin{enumerate}[label=(\alph*)]
            \item $T^{z}:=T(z, \cdot):\mathcal{X}\times [0,1]\to \mathcal{X}\times [0,1]$ is a Borel isomorphism, i.e. $T^{z}$ is Borel measurable, bijective and its inverse is also Borel measurable,
            \item for the Lebesgue measure $\lambda$ on $[0,1]$, $(\id, T^{z})_{\#}(\mu^{z}\otimes \lambda)\in \cpl(\mu^{z}\otimes \lambda, \nu^{z}\otimes \lambda)$ and its projection on $\mathcal{X}\times \mathcal{X}$ is $\pi^{z}$.
        \end{enumerate}
        Now let us choose a partition $\mathcal{R}_n$ of $\mathcal{X}$ that consists of at most countable sets and satisfies $\lim_{n\to \infty}\norm{\mathcal{R}_n}=0$. We apply \cite[Proposition 3.25]{beiglbock2022denseness} to find Borel measurable maps $\Phi_n:\mathcal{Z}\times \mathcal{X}\to \mathcal{X}\times [0,1]$ such that
        \begin{enumerate}[label=(\alph*)]
        \setcounter{enumi}{2}
            \item $\Phi_n^{z}:=\Phi_n(z, \cdot):\mathcal{X}\to \mathcal{X}\times [0,1]$ is Borel isomorphism,
            \item $(\Phi^z_n)_{\#}(\mu^z|_A)=(\mu^z|_{A})\otimes \lambda$ for all $A\in \mathcal{R}_n$ where $\mu^z|_A(\cdot):=\mu^z(\,\cdot \cap A)$.
        \end{enumerate}
        Let $P_{\mathcal{X}}$ and $P_{\mathcal{Z}}$ be projection maps on $\mathcal{X}$ and $\mathcal{Z}$ respectively and define a map $T_n:\mathcal{Z}\times \mathcal{X}\to \mathcal{X}$ via $T_n = P_{\mathcal{X}}\circ T \circ (P_{\mathcal{Z}}, \Phi_n)$. Note that $\pi^{z}_n:=(\id, T^z_n)_{\#}\mu^{z}\in \cpl(\mu^z, \nu^z)$ by $\text{(b)}$ and $\text{(d)}$. Similarly as in the proof of \cite[Theorem 2.6]{beiglbock2022denseness}, we can show that
        \begin{align}
            \pi^z_n(A\times B)
            =\pi^z(A\times B) \text{ for any } A, B\in \mathcal{R}_n \text{ and } z\in \mathcal{Z}.
        \end{align}
        Indeed, for $A, B\in \mathcal{R}_n$,
        \begin{align}
            \pi^z_n(A\times B)
            =\mu^z(A\cap (T^z_n)^{-1}(B))
            &=\mu^z|_{A}((\Phi^z_n)^{-1}\circ (T^z)^{-1}\circ (P_{\mathcal{X}})^{-1}(B))\\
            &=\mu^z|_{A}\otimes \lambda ((T^z)^{-1}\circ (P_{\mathcal{X}})^{-1}(B))\\
            &=\mu^z\otimes \lambda ((A\times [0,1])\cap (T^z)^{-1}(B\times [0,1]))\\
            &=(\id, T^z)_{\#}(\mu^z\otimes \lambda)((A\times [0,1])\times (B\times[0,1]))\\
            &=\pi^z(A\times B).
        \end{align}
        The first equality comes from the definition of $\pi_n^z$. For the second equality, we use the definition of $T^z_n$. The property $\text{(d)}$ implies the third equality. The fourth and the fifth equality are straightforward from the pushforward operator and the last equality is obtained from the property $\text{(b)}$. Thus, we deduce from \cite[Lemma 2.4]{beiglbock2022denseness} that $\pi^z_n \to \pi^z$ in $\w_p$ on $(\mathcal{X}\times \mathcal{X})^2$. This implies the desired results.
    \end{proof}

\section*{Acknowledgements}
J. Blanchet acknowledges support from NSF grant 2312204 and ONR N000142412655.


\begin{small}
\bibliographystyle{abbrv}
\bibliography{smooth_aw2024}
\end{small}

\end{document}